\documentclass[11pt]{amsart}

\usepackage{layout}
\usepackage[american]{babel}
\usepackage{amssymb}
\usepackage{graphicx}
\usepackage{amscd}
\usepackage{mathtools}
\usepackage[mathscr]{eucal}
\usepackage[all,knot,arc]{xy}
\usepackage{hyperref}
\usepackage{tikz}

\usepackage[normalem]{ulem}

\usepackage{xcolor}

\newtheorem{theorem}{Theorem}[section]
\newtheorem{proposition}[theorem]{Proposition}
\newtheorem{propn}[theorem]{Proposition}
\newtheorem{corollary}[theorem]{Corollary}

\newtheorem{lemma}[theorem]{Lemma}
\theoremstyle{definition}
\newtheorem{definition}[theorem]{Definition}

\theoremstyle{remark}
\newtheorem{remark}[theorem]{Remark}

\newcommand{\olink}{\mathbb{L}} % multi-based link
\newcommand{\oknot}{\mathbb{K}} % doubly-based knot
\newcommand{\ounknot}{\mathbb{U}_1} % doubly-based knot
\newcommand{\dlink}{\mathcal{L}} % disoriented link
\newcommand{\du}{\mathcal{U}_1} % disoriented unknot
\newcommand{\bdlink}{\mathscr{L}} % bipartite disoriented link
\newcommand{\w}{\mathbf{w}}
\newcommand{\z}{\mathbf{z}}
\newcommand{\p}{\mathbf{p}}
\newcommand{\q}{\mathbf{q}}
\renewcommand{\l}{\mathbf{l}}
\renewcommand{\o}{\mathbf{o}}
\newcommand{\dlcob}{\mathbb{S}}
\newcommand{\dcob}{\mathcal{S}}

\newcommand{\dec}{\mathcal{A}}

\newcommand{\motion}{\mathcal{M}}

\newcommand{\m}{\overline}

\newcommand{\gdlcob}{\mathbb{G}}
\newcommand{\gdcobu}{{\mathcal{G}^u}}
\newcommand{\gdcobo}{{\mathcal{G}^o}}
\newcommand{\tdcobu}{\mathcal{T}^u}
\newcommand{\tdcobo}{\mathcal{T}^o}

\DeclareMathOperator{\id}{id}

\newcommand{\bun}{u_b}
\newcommand{\buln}{ul_b}
\newcommand{\ubun}{u^u_b}
\newcommand{\ubuln}{ul^u_b}

\newcommand{\flip}{\mathcal{F}}

\newcommand{\T}{\mathbb{T}}
\renewcommand{\a}{\alpha}

\newcommand{\ba}{{\boldsymbol{\alpha}}}
\newcommand{\bb}{{\boldsymbol{\beta}}}

\newcommand{\x}{\textbf{x}}
\newcommand{\y}{\textbf{y}}
\newcommand{\de}{\partial}

\newcommand{\RedModSpace}{\overline{\mathcal{M}}}

\newcommand{\F}{\mathbb{F}}

\newcommand{\CFK}{\textnormal{CFK}}
\newcommand{\HF}{\textnormal{HF}}

\newcommand{\HFK}{\textnormal{HFK}}

\newcommand{\CFLu}{\textnormal{CFL}'}
\newcommand{\CFKu}{\textnormal{CFK}'}
\newcommand{\CFLum}{\CFLu}
\newcommand{\CFKum}{\CFKu}
\newcommand{\HFLu}{\textnormal{HFL}'}
\newcommand{\HFKu}{\textnormal{HFK}'}
\newcommand{\HFKum}{\HFKu}
\newcommand{\HFLum}{\HFLu}
\newcommand{\CFLc}{\mathcal{CFL}}

\newcommand{\CFLcm}{\CFLc^-}

\newcommand{\Zmap}{F^Z} % Zemke's map
\newcommand{\Hmap}{F^F} % Haofei's map

\newcommand{\Tors}{\textnormal{Tors}}

\DeclareMathOperator{\Order}{Ord}
\newcommand{\OrdU}{\Order_U}
\newcommand{\OrdI}{\Order_I}
\newcommand{\Ord}{\Order}

\newcommand{\gr}{\mathbf{gr}}

\newcommand{\set}[1]{\left\{#1\right\}}
\newcommand{\mc}{\mathcal}
\newcommand{\sm}{\setminus}
\newcommand{\Z}{\mathbb{Z}}

\DeclarePairedDelimiter{\ceil}{\lceil}{\rceil}
\DeclarePairedDelimiter{\floor}{\lfloor}{\rfloor}

\usepackage{amsmath}
\usepackage{amsxtra}
\usepackage{amscd}
\usepackage{amsthm}
\usepackage{amsfonts}
\usepackage{amssymb}
\usepackage{eucal}
\usepackage{graphicx}
\usepackage{wrapfig}
\usepackage{hyperref}
\usepackage{enumerate}
\usepackage{scalerel}
\usepackage{setspace}
\graphicspath{ {images/} }

\usepackage{caption}
\usepackage{xcolor}
\definecolor{mygrey}{gray}{0.45}
\definecolor{blue}{RGB}{63, 105, 225}
\definecolor{pink}{RGB}{180, 105, 90}
\definecolor{red}{RGB}{180,0,0}
\definecolor{green}{RGB}{60, 120, 60}

\theoremstyle{definition}

\newtheorem{claim}{Claim}

\newcommand\nc{\newcommand}
\nc\Span{\text{\rm Span}}
\nc\Id{\text{Id}}
\nc\coker{\text{coker}}

\nc \cc {\mathbb{C}}
\nc \dd {\mathbb{D}}
\nc \ff {\mathbb{F}}
\nc \hh {\mathbb{H}}
\nc \ii {\mathbb{I}}
\nc \jj {\mathbb{J}}
\nc \kk {\mathbb{K}}
\nc \LL {\mathbb{L}} % can't have \ll because that's "much less than"
\nc \mm {\mathbb{M}}
\nc \nn {\mathbb{N}}
\nc \oo {\mathbb{O}}
\nc \pp {\mathbb{P}}
\nc \qq {\mathbb{Q}}
\nc \rr {\mathbb{R}}

\nc \zz {\mathbb{Z}}

\nc\cA{\mathcal{A}}
\nc\cB{\mathcal{B}}
\nc\cC{\mathcal{C}}
\nc\cD{\mathcal{D}}
\nc\cE{\mathcal{E}}
\nc\cF{\mathcal{F}}
\nc\cG{\mathcal{G}}
\nc\cH{\mathcal{H}}
\nc\cK{\mathcal{K}}
\nc\cL{\mathcal{L}}
\nc\cM{\mathcal{M}}
\nc\cN{\mathcal{N}}
\nc\cO{\mathcal{O}}
\nc\cP{\mathcal{P}}
\nc\cQ{\mathcal{Q}}
\nc\cR{\mathcal{R}}
\nc\cS{\mathcal{S}}
\nc\csf{\mathcal{S}\mathcal{F}}

\nc\fg{\mathfrak{g}}
\nc\fm{\mathfrak{m}}
\nc\fp{\mathfrak{p}}
\nc\fso{\mathfrak{so}}
\nc\fsu{\mathfrak{su}}

\nc\Tr{\text{Tr}}
\nc\into{\hookrightarrow}

\nc\st{\text{ s.t. }}
\nc\intense[1]{\textcolor[rgb]{1.00,0.00,0.00}{\textbf{#1}}}

\nc\Mat{\text{\rm Mat}}
\nc\GL{\text{\rm GL}}
\nc\SU{\text{\rm SU}}
\nc\SO{\text{\rm SO}}
\nc\SL{\text{\rm SL}}
\nc\Sp{\text{\rm Sp}}
\nc\EL{\text{\rm EL}}

\nc\GEM{\text{\rm GEM}}
\nc\Alt{\text{\rm Alt}}
\nc\Sym{\text{\rm Sym}}
\nc\Hess{\text{Hess}}

\nc\Crit{\text{Crit}}

\nc\inject{\hookrightarrow}

\nc{\mattwo}[4]{\left[\begin{array}{cc} #1  & #2\\  #3 & #4 \\ \end{array} \right]}
\nc{\matthree}[9]{\left[\begin{array}{ccc} #1  & #2 & #3\\  #4 & #5 & #6 \\ #7 & #8 & #9 \\ \end{array} \right]}
\nc{\vecttwo}[2]{\left[\begin{array}{c} #1 \\ #2 \\ \end{array} \right]}
\nc{\vectthree}[3]{\left[\begin{array}{c} #1 \\ #2\\  #3 \\ \end{array} \right]}

\nc{\del}{\partial}
\nc\onto{\twoheadrightarrow}

\nc\const{\text{const}}

\nc\rrp{\rr P}
\nc\ul{\underline}
\nc\ol{\overline}

\nc\oline{\overline}
\nc\oset{\overset}
\nc\uset{\underset}

\nc\heart{\heartsuit}
\nc\spade{\spadesuit}
\nc\club{\clubsuite}

\nc\marg[1]{\marginnote{\boxed{\text{#1}}}}
\nc\margq[1]{\marginnote{\textcolor[rgb]{1.00,0.00,0.00}{#1}}}

\nc\Char{\text{Char}}
\nc\Frac{\text{Frac}}
\nc\wo{\backslash}
\nc\diag{\text{diag}}
\nc\wtl{\widetilde}
\nc\nsubgp{\triangleleft}

\nc\Cay{\text{Cay}}
\nc\Hom{\text{Hom}}
\nc\Gp{\text{Gp}}
\nc\Set{\text{Set}}
\nc\la{\langle}
\nc\ra{\rangle}

\nc\Spec{\text{Spec}}
\nc\ad{\text{ad}}

\nc\wht{\widehat}
\nc\ddx[2]{\frac{\partial {#1}}{\partial {#2}}}
\nc\dddx[3]{\frac{\partial^2 {#1}}{\partial {#2}\partial{#3}}}
\nc\mult{\text{mult}}
\nc\supp{\text{supp}}
\nc\sign{\text{sign}}

\nc\tr{\text{tr}}
\nc\stab{\text{stab}}
\nc\im{\text{im}}

\nc\sech{\text{sech}}

\nc\Ind{\text{Ind}}
\nc\tsf{\text{sf}}
\nc\End{\text{End}}
\nc\Hol{\text{Hol}}
\nc\hol{\text{hol}}
\nc\Lie{\text{Lie}}

\nc\vol{\text{vol}}
\nc\ind{\text{ind}}
\nc\Met{\mathcal{M}\text{et}}
\nc\Grass{\mathcal{G}\text{rass}}

\nc\longto{\longrightarrow}
\nc\grad{\text{grad}}
\nc\Map{\text{Map}}
\nc\Indx{\text{Ind}}
\nc\indx{\text{ind}}
\nc\pt{\text{pt}}
\nc\Aut{\text{Aut}}
\nc\Br{\text{Br}}
\nc\Gr{\text{Gr}}
\nc\CS{\text{CS}}

\nc\Proj{\text{Proj}}
\nc\Ad{\text{Ad}}
\nc\Diff{\text{Diff}}
\nc\sing{\text{sing}}
\nc\order{\text{order}}
\nc\bsl{\backslash}
\nc\dom{\text{dom}}

\nc\smallcoprod{\scaleobj{0.7}\coprod}

\makeatletter
\@namedef{subjclassname@2020}{2020 MSC}
\makeatother

\begin{document}

\title[Non-orientable cobordisms and torsion order]{Non-orientable link cobordisms\\ and torsion order in Floer homologies}

\author[Sherry Gong]{Sherry Gong}
\address{Texas A{\&}M University \\ College Station, TX 77843}
\email{\href{mailto:sgongli@math.tamu.edu}{sgongli@math.tamu.edu}}
\urladdr{\url{http://www.people.tamu.edu/~sgongli}}

\author[Marco Marengon]{Marco Marengon}
\address{Alfr\'ed R\'enyi Institute of Mathematics \\ 1053 Budapest, Hungary}
\email{\href{mailto:marengon@renyi.hu}{marengon@renyi.hu}}
\urladdr{\url{https://users.renyi.hu/~marengon/}}

% \date{\today}

\begin{abstract}
We use unoriented versions of instanton and knot Floer homology to prove inequalities involving the Euler characteristic and the number of local maxima appearing in unorientable cobordisms, which mirror results of a recent paper by Juhasz, Miller, and Zemke concerning orientable cobordisms.
Most of the subtlety in our argument lies in the fact that maps for non-orientable cobordisms require more complicated decorations than their orientable counterparts.
We introduce unoriented versions of the band unknotting number and the refined cobordism distance and apply our results to give bounds on these based on the torsion orders of the Floer homologies.
Finally, we show that the difference between the unoriented refined cobordism distance of a knot $K$ from the unknot and the non-orientable slice genus of $K$ can be arbitrarily large.
\end{abstract}

\maketitle

\section{Introduction}

A classical problem in low-dimensional topology is the study of embedded orientable surfaces in 4-manifolds.
The special case of surfaces with boundary has been a particularly popular topic for a very long time, and it includes for example questions pertaining to the slice genus of a knot or the complexity of a knot or link cobordism.
%In a recent paper, Juh\'asz, Miller, and Zemke proved an inequality relating the genus of an orientable knot cobordism and the number of local maxima to the torsion order in knot Floer homology~\cite{JMZ}.

On the other hand, the study of \emph{non-orientable} surfaces and knot cobordisms in $I \times S^3$ has received increasing attentions in the last decade~\cite{Batson, Upsilon, GM, Fan}, and there are now several bounds to the non-orientable slice genus of a knot. However, if a knot bounds a non-orientable surface of a given `genus', it is not clear how complicated the embedding must be. In this paper we tackle this question by proving a non-orientable analogue of a recent result of Juh\'asz--Miller--Zemke.
%\MM{Is it true that there are no lower bounds on the complexity?}
In a recent paper~\cite{JMZ}, they proved an inequality involving the number of local maxima and the genus appearing in an oriented knot cobordism using a version of knot Floer homology.
In this paper, we prove similar inequalities for non-orientable knot cobordisms using the torsion orders of unoriented versions of knot Floer homology and instanton Floer homology.

As for knot Floer homology, we use Ozsv\'ath--Stipsicz--Szab\'o's \emph{unoriented knot Floer homology} $\HFKum$~\cite{tHFK}, which is a module over $\F[U]$. For a knot $K \subset S^3$ we define its unoriented knot Floer torsion order as
\[
\OrdU(K)= \min\set{n \geq 0 \,\middle|\, U^n \cdot \Tors = 0},
\]
where $\Tors \subset \HFKum(K)$ denotes the $\F[U]$-torsion subgroup.

In the instanton setting, we use Kronheimer and Mrowka's \emph{instanton Floer homology with local coefficients}, denoted by $I^\sharp(K)$,
which is a module over a Noetherian domain $\cS$ which has a special element $P$~\cite{km_iabnh}. We will restrict our attention to certain domains $\cS$, for which $I^\sharp(K)$ is functorial for nonorientable knot cobordisms with singular bundles represented by surfaces $\omega$ with $\partial \omega$ on the cobordism. In this case, it can be shown that for a knot $K$ and for the torsion part $\Tors$ of $I^\sharp(K)$, there is a positive integer $n$ such that $P^n\cdot\Tors = 0$. Thus, we define
\[
\OrdI(K) = \min\set{n\geq0 \, \middle| \, P^n \cdot \Tors  = 0}.
\]

%%%%%% Sept 20 Change 1

For an unorientable surface $\Sigma$ with $n$ boundary components, recall that its \emph{non-orientable genus} is
\[
\gamma(\Sigma) = 2 - \chi(\Sigma) -n.
\]
For example, $\mathbb{RP}^2$ (with an arbitrary number of punctures) has non-orientable genus $1$.
%We state our results in terms of the Euler characteristic of the surface rather than its non-orientable genus $\gamma$. To translate between the two notations, recall that for a non-orientable surface with $n$ punctures
%\[\gamma(\Sigma) = 2 - \chi(\Sigma) -n.\]
%It is often useful to consider genus-0 \emph{orientable} surfaces as well (e.g., knot concordances), for which we set $\gamma = 0$.
Note that for unorientable knot cobordisms $\gamma(\Sigma) = -\chi(\Sigma)$. With this notation in mind, we state our main theorem.

%\MM{Do we want to talk about crosscaps from the very beginning?}

\begin{theorem}
\label{thm:main}
Let $K_1$ and $K_2$ be knots in $S^3$. Suppose that there is an unorientable knot cobordism $\Sigma$ in $I \times S^3$ from $K_1$ to $K_2$ with $M$ local maxima. Then
\begin{equation}
\label{eq:I_main}
\OrdI(K_1) \leq \max\set{\OrdI(K_2), M} + \gamma(\Sigma)
%\OrdI(K_1) \leq \max\set{\OrdI(K_2), M} - \chi(\Sigma)
\end{equation}
and
\begin{equation}
\label{eq:U_main}
\OrdU(K_1) \leq \max\set{\OrdU(K_2), M} + \gamma(\Sigma).
%\OrdU(K_1) \leq \max\set{\OrdU(K_2), M} - \chi(\Sigma) +1.
\end{equation}
\end{theorem}

From a formal viewpoint, Theorem \ref{thm:main} is analogous to \cite[Theorem 1.1]{JMZ}. The proof of Theorem \ref{thm:main} uses the functorial properties of $\HFKum$ and $I^\sharp$ (see \cite{Fan, km_iabnh}), in a similar way as \cite[Theorem 1.1]{JMZ} relies on knot Floer homology cobordism maps~\cite{JFunctoriality, ZFunctoriality}.
%%%%% Sept 20 citation added above
Despite being inspired by \cite[Theorem 1.1]{JMZ}, the proof of Theorem \ref{thm:main} must necessarily deviate from it. Recall that in order to define a cobordism map in knot Floer homology one needs to choose a properly embedded 1-manifold on the surface, often called the set of \emph{decorations}. In \cite{JMZ} the chosen decorations were a pair of parallel arcs, which make the computations of the cobordism maps more tractable. This choice does not work for non-orientable cobordisms in $\HFKum$, so we are forced to choose different decorations, which make the cobordism map more complicated. To circumvent this problem we related the resulting unorientable cobordism to an orientable one, then use a stabilization lemma proved by Ian Zemke (cf.~Lemma \ref{lem:Ian}). In the case of $I^\sharp$, for a cobordism $\Sigma$ to define a map, one needs a surface $\omega$ with boundary $\partial \omega \subset \Sigma$. The natural choice for orientable cobordisms would be $\omega = \emptyset$, in which case \cite{JMZ} applies verbatim to the case of $I^\sharp$. The map can be defined for the non-orientable surfaces we are interested in, but it will usually vanish. To overcome this problem, we choose a particular $\omega$ that allows us to control the induced map. 
%%%%%% Sept 20 change 3

\begin{remark}
\label{rem:all}
Note that while Theorem \ref{thm:main} is stated for unorientable cobordisms, both inequalities also hold for \emph{orientable} cobordisms too.
The proof follows verbatim from \cite{JMZ}, by replacing knot Floer homology with the desired Floer theory.
\end{remark}

We prove several applications of Theorem \ref{thm:main}, which mirror those of \cite[Theorem 1.1]{JMZ}.

\subsection{Unorientable ribbon cobordism}
A knot cobordism in $I \times S^3$ is called \emph{ribbon} if it has no local maxima. For example, a ribbon concordance (i.e., a cobordism of genus $0$) from the unknot to a knot $K$ is equivalent to a a ribbon disc for $K$. Theorem \ref{thm:main} has a straightforward application to non-orientable ribbon cobordisms.

\begin{corollary}
\label{cor:ribbon}
Let $K_1$ and $K_2$ be knots in $S^3$. Suppose that there is an unorientable ribbon cobordism $\Sigma$ in $I \times S^3$ from $K_1$ to $K_2$. Then
\[
\OrdI(K_1) \leq \OrdI(K_2) + \gamma(\Sigma)
\]
and
\[
\OrdU(K_1) \leq \OrdU(K_2) +\gamma(\Sigma).
\]
\end{corollary}

\subsection{The unorientable refined cobordism distance}
The standard cobordism distance between two knots $K_1$ and $K_2$ is given by $d_o(K_1, K_2) = 2g_4(K_1\#\m{K_2})$, where $g_4$ denotes the standard slice genus. This is not a distance on the set of knots, because concordant knots have distance $0$, but it descends to a well defined distance on the concordance group~\cite{Baader}.
In \cite{JMZ}, Juh\'asz, Miller, and Zemke define a refined cobordism distance on the set of knots, and give lower bounds to it in terms of the torsion order in knot Floer homology.

There are analogous unorientable notions too.
For an (orientable or unorientable) cobordism $\Sigma$ in $I \times S^3$ from $K_1$ to $K_2$ with $m$ local minima and $M$ local maxima, let
\[
|\Sigma| = \max\set{m,M} -\chi(\Sigma).
\]
\begin{definition}
\label{def:refcobdist}
Given knots $K_1, K_2, \subset S^3$, we define the \emph{standard unorientable cobordism distance} $d_u$ and the \emph{refined unorientable cobordism distance} $d_u^r$ between them as
\[
d_u(K_1, K_2) = \min\set{-\chi(\Sigma)}
\qquad \text{and} \qquad
d_u^r(K_1, K_2) = \min\set{|\Sigma|},
\]
where in both cases $\Sigma$ varies over all unorientable connected cobordisms and all genus-0 orientable cobordisms (i.e., concordances).
\end{definition}

\begin{remark}
The orientable counterparts, the \emph{standard orientable cobordism distance} $d_o$ from \cite{Baader} and the \emph{refined orientable cobordism distance} $d_o^r$ from \cite{JMZ} are defined in the same way as in Definition \ref{def:refcobdist}, but the surface $\Sigma$ now varies over all \emph{orientable} connected cobordisms.
One can also define analogous notions $d_a$ and $d_a^r$, which we can call \emph{all-surface cobordism distances}, where $\Sigma$ varies over all (orientable or unorientable) connected cobordisms.
\end{remark}

It is immediate to see that $d_o$, $d_u$, and $d_a$ are distances on the concordance group and $d_o^r$, $d_u^r$, and $d_a^r$ are distances on the set of knots.

Theorem \ref{thm:main} implies the following lower bounds.

\begin{corollary}
\label{cor:dur}
If $K_1$ and $K_2$ are knots in $S^3$, then
\[
|\OrdI(K_1) - \OrdI(K_2)| \leq d_u^r(K_1,K_2)
\]
and
\[
|\OrdU(K_1) - \OrdU(K_2)| \leq d_u^r(K_1,K_2).
\]
\end{corollary}

In view of Remark \ref{rem:all}, one can in fact replace $d_u^r$ with $d_a^r$. However, for orientable cobordisms one can also use the standard versions of instanton and knot Floer homology, which should give better bounds.

We use Corollary \ref{cor:dur} to show that the difference between $d_u^r(K_1,K_2)$ and $d_u(K_1,K_2)$ can be arbitrarily large.
\begin{corollary}
\label{cor:example}
For all $\gamma\geq1$ and $m\geq1$, there exists a knot $K_{\gamma, m}$ with $d_u(K_{\gamma, m}, U_1) = \gamma_4(K_{\gamma, m}) = \gamma$, and such that $d_u^r(K_{\gamma, m}, U_1) \geq \gamma+m$.

Thus, each unorientable surface $\Sigma \subset B^4$ with $\de\Sigma = K_{\gamma, m}$ and $\gamma(\Sigma) = \gamma$ has at least $m$ local minima (with respect to the radial function).
\end{corollary}
The knots $K_{\gamma, m}$ that we consider in the proof of Corollary \ref{cor:example} are a subfamily of torus knots, for which $\OrdU$ can be computed explicitly.

\subsection{The unoriented band-unlinking number}
%\MM{Here I am using the expression ``unoriented'' to mean either orientable or not. Better suggestions?}

For a knot $K$ in $S^3$, the \emph{oriented band-unknotting number} $\bun(K)$ is defined as the minimum number of oriented band surgeries that turn $K$ into the unknot. This was called $SH(2)$-unknotting number in \cite{HNT}.
Its unoriented counterpart $\ubun(K)$, called $H(2)$-unknotting number in \cite{HNT}, seems to pre-date $\bun(K)$ in the literature, since Lickorish proved that there exist knots with $\ubun(K)>1$ in \cite{LickorishUnknotting}. Note that in the definition of $\ubun(K)$ we allow both orientable and unorientable band surgeries.

Juh\'asz--Miller--Zemke~\cite{JMZ} introduced a variation, called the \emph{oriented band-unlinking number} $\buln(K)$, is defined as the minimum number of oriented band surgeries that turn $K$ into an unlink. Of course $\buln(K) \leq \bun(K)$, and they proved that $\Order_v(K) \leq \buln(K)$ for all knots $K$ in $S^3$.
Using Theorem \ref{thm:main} we can derive a similar result for the corresponding unoriented notion.

\begin{definition}
The \emph{unoriented band-unlinking number} $\ubuln(K)$ of a knot $K$ in $S^3$ is defined as the minimum number of (orientable or unorientable) band surgeries that turn $K$ into an unlink.
\end{definition}

Clearly, we have
\begin{align*}
\ubuln(K) &\leq \ubun(K)\\
\rotatebox[origin=c]{-90}{$\leq$} \hspace{2ex} & \hspace{5ex} \rotatebox[origin=c]{-90}{$\leq$} \\
\buln(K) &\leq \bun(K)
\end{align*}

\begin{corollary}
\label{cor:ubuln}
For a knot $K$ in $S^3$,
\[
\OrdI(K) \leq \ubuln(K) \qquad \text{and} \qquad \OrdU(K) \leq \ubuln(K).
\]
\end{corollary}
%\begin{proof}
%The corollary follows from Theorem \ref{thm:HFK} (and Remark \ref{rem:comparison}) in the same way as \cite[Corollary 1.7]{JMZ} follows from \cite[Theorem 1.2]{JMZ}.
%\end{proof}

%\begin{remark}
%Using Remark \ref{rem:comparison}, in the orientable case one can also prove that $\OrdU(K) \leq \buln(K)$, without the $+1$ on the right hand side. However, since $\OrdU(K) \leq \Order_{u,v}^{\mathrm{Chain}}(K)$, this bound can never be better than the one in \cite[Proposition 7.4.(1)]{JMZ}.
%\end{remark}

\begin{remark}
\label{rem:Wong}
In private communication~\cite{Wong}, Wong has informed us of a proof, using methods analogous to \cite{AE18}, that if there is a cobordism $\Sigma \subset I \times S^3$ from $K_1$ and $K_2$ (no matter whether orientable or unorientable) with $m$ minima, $b$ saddles, and $M$ maxima, then
\[
|\OrdU(K_1) - \OrdU(K_2)| \leq m+b+M.
\]
Since the unlink has vanishing torsion order, this would recover the inquality of Corollary \ref{cor:ubuln} involving $\OrdU$.
\end{remark}

Our paper is one of several recent papers related to ribbon cobordisms.
In \cite{Zemke_obstructs_ribbon}, Zemke showed that knot Floer homology obstructs ribbon concordance, a result that prompted a flurry of interesting results in this area, including \cite{Levine_Zemke}, \cite{Miller_Zemke}, \cite{DLVVW}, \cite{Kang}, and \cite{CGLLSZ}.
Other papers in the area are Sarkar's paper on the ribbon distance~\cite{Sarkar_ribbon} and the already cited paper of Juh\'asz, Miller, and Zemke~\cite{JMZ}, which is the closest paper to ours.
%In \cite{Zemke_obstructs_ribbon}, Zemke showed that knot Floer homology obstructs ribbon concordance, a result that prompted a flurry of interesting results in this area: in \cite{Levine_Zemke}, Levine and Zemke showed the analogous result for Khovanov homology, in \cite{LVVW}, Lidman, Vela-Vick and Wong showed it for the Heegaard Floer homology of the double branched cover, and in \cite{Kang}, Kang showed this result for a more general class of of link TQFTs. In \cite{Miller_Zemke}, Miller and Zemke generalize the result to strongly homotopy-ribbon concordances. 

\subsection*{Organisation}
The first two sections of the paper are on instanton Floer homology: we review the necessary background in Section \ref{sec:bgI}, and we prove the main instanton technical result (Proposition \ref{prop:I}) in Section \ref{sec:techI}.
In the following two sections we do the same for knot Floer homology: after a review in Section \ref{sec:bgHFK}, we prove the main knot Floer technical result (Proposition \ref{prop:HFK}) in Section \ref{sec:techHFK}.
In Section \ref{sec:corollaries} we prove Theorem \ref{thm:main} and the applications discussed in the introduction (Corollaries \ref{cor:ribbon}, \ref{cor:dur}, and \ref{cor:ubuln}).
Finally, in Section \ref{sec:examples} we compute the torsion order $\OrdU$ for a subfamily of torus knots and prove Corollary \ref{cor:example}.

\subsection*{Acknowledgements}
We warmly thank Haofei Fan for his support and help. Special acknowledgements go to Ciprian Manolescu, who suggested to investigate unorientable cobordisms, and to C.-M.~Michael Wong, who shared his work with us and who gave us helpful comments on an early draft of the paper.
We are also extremely grateful and indebted to Ian Zemke, who proved Lemma \ref{lem:Ian}.
Finally, we would like to thank the anonymous referee for their careful reading and helpful comments.
MM was partially supported by NSF FRG Grant DMS-1563615 and the Max Planck Institute for Mathematics.
SG was partially supported by NSF DMS-2055736.

\section{Background on instanton homology with local systems}
\label{sec:bgI}

\subsection{Instanton homology groups}

Kronheimer and Mrowka introduced singular instanton homology with local systems in \cite{km_khgfi}, and introduced several more involved variants of it in \cite{km_iabnh}. In this paper we will be working with a variant from the latter. Let us now review the relevant definitions and properties, following \cite{km_iabnh} and \cite{km_iasciok}.

Let $Y$ be a closed, oriented 3-manifold, let $L$ be a link in $Y$ and let $y_0$ be a basepoint in $Y$, and let $B_{y_0}$ be a ball around $y_0$ that does not intersect $L$. Let $\theta_0 \subset Y$ be a standard $\theta$-web in $B_{y_0}$.
Let $\omega$ be a 1-dimensional submanifold of $Y$, which consists of components that are circles disjoint from $L$ and $B_{y_0}$ and arcs which have endpoints on $L$ and are otherwise disjoint from $L$.

Then there is an associated space, $\cB^\sharp(Y,L)_\omega$ of $SO(3)$ connections on $Y$ that are singular at $L \cup \theta_0$, which lift to $SU(2)$ away from the $L \cup \omega \cup \theta_0$, such that the $SU(2)$ holonomy around $\omega$ is $-1$ and the $SU(2)$ holonomies around components of $L$ and arcs of $\theta_0$ are conjugate to $I \in SU(2)$, when we regard $SU(2)$ as unit quaternions and $1,I,J,K$ are the fundamental quaternion units.

The local system $\Gamma$ is defined using three maps $h_i:\cB^\sharp(Y,L)_\omega \to \rr/\zz$, for $i = 1,2,3$, given by taking holonomy along the three arcs of the $\theta$ web, which gives three maps to $SU(2)$ and then composing with a character $SU(2) \to U(1) = \rr/\zz$ to get maps to $\rr/\zz$. Let $\cR = \ff_2[\zz^3]$ be the group ring, which we can also write as the ring of Laurent polynomials in three variables,
\[\cR = \ff_2[T_1^{\pm 1},T_2^{\pm 1},T_3^{\pm 1}].\]
Then $\Gamma$ is defined as the pull back via $(h_1,h_2,h_3)$ of a particular local system over $(\rr/\zz)^3$ with fibre the free rank $1$ module over $\cR$.
For a commutative ring $\cS$ and a homomorphism $\sigma:\cR \to \cS$, let $\Gamma_\sigma$ denote the induced local system of $\cS$ modules. 

The instanton homology group $I^\sharp(Y,L;\Gamma_\sigma)_\omega$ is defined as the Floer homology of $\cB^\sharp(Y,L)_\omega$ with a perturbed Chern Simons functional and with the local system $\Gamma_\sigma$. (In \cite{km_iabnh}, there is an additional map $h_0:\cB^\sharp(Y,L)_\omega \to \rr/\zz$ coming from taking holonomy along the link itself, and $\cR$ is defined to be $\ff_2[\zz^4]$, but for our purposes, we will only be using the local system coming from $h_1,h_2$, and $h_3$.)

\subsection{Maps induced by cobordisms with dots}

We now review the funcatoriality of $I^\sharp(Y,L;\Gamma_\sigma)_\omega$. Keeping previous notation, let $\sigma:\cR \to \cS$ be a map of commutative rings.

For $i = 1,2$, let $Y_i$ denote a closed, oriented 3-manifolds, with a link $L_i$ and a 1-manifold $\omega_i$ embedded in $Y_i$ with boundary on $L_i$ and otherwise not intersecting $L_i$.

For a cobordism of pairs $(X,S)$ from $(Y_1,L_1)$ to $(Y_2,L_2)$, and $\omega$ a $2$ manifold with corners, whose boundary pieces are $\omega_1$ and $\omega_2$ in $Y_1$ and $Y_2$ respectively, together with arcs and circles in $S$, then there is an induced map
\[I^\sharp(X,S;\Gamma_\sigma)_\omega:I^\sharp(Y_1,L_1;\Gamma_\sigma)_{\omega_1} \to I^\sharp(Y_2,L_2;\Gamma_\sigma)_{\omega_2}\]
of $\cS$ modules. 

This functoriality can be extended to morphisms given by cobordisms of pairs with dots on the surfaces. That is, for a cobordism of pairs $(X,S)$, define a \textit{dot} on $S$ to be an interior point $p \in S$ along with an orientation of $T_pS$. Then for dots $p_1, p_2, \ldots, p_d$ on $S$, there is an induced map of $\cS$ modules
\[I^\sharp(X,S, p_1, p_2, \ldots, p_d;\Gamma_\sigma)_\omega:I^\sharp(Y_1,L_1;\Gamma_\sigma)_{\omega_1} \to I^\sharp(Y_2,L_2;\Gamma_\sigma)_{\omega_2}.\]

In our computations, we will always have $Y_1,Y_2 = S^3$ and $X = S^3 \times [0,1]$. Moreover we will be using the same $\Gamma_\sigma$. Thus, we will denote our cobordisms by
\[I^\sharp(S, p_1, p_2, \ldots, p_d)_\omega =I^\sharp(X,S, p_1, p_2, \ldots, p_d;\Gamma_\sigma)_\omega\]

\subsection{Properties of the induced maps}
\label{sec:I_properties}

Before going over some of the properties of the maps of $\cS$ modules induced by cobordisms, let us recall two particular elements of the rings $\cR$ and $\cS$. Writing $\cR = \ff_2[T_1^{\pm1},T_2^{\pm1},T_3^{\pm1}]$, the elements $P$ and $Q$ are given by 
\[P = T_1T_2T_3 + T_1T_2^{-1}T_3^{-1} + T_1^{-1}T_2T_3^{-1}+ T_1^{-1}T_2^{-1}T_3\]
and
\[Q = T_1^2 + T_1^{-2}+T_2^2 + T_2^{-2}+T_3^2 + T_3^{-2}.\]

For $\sigma:\cR \to \cS$, the elements $\sigma(P),\sigma(Q) \in \cS$ will also be denoted $P,Q$ respectively.

% We will be using the following properties of induced maps, which are shown \cite{km_iasciok} 
\begin{enumerate}
\item \label{it:neckcutting} (Lemma 3.2 of \cite{km_iasciok}) Let $S$ be an oriented cobordism. Suppose $S'$ is obtained from $S$ by adding an internal $1$-handle connecting points $p,q \in S$, where $p,q$ both have the same orientation as $S$. Then
\[I^\sharp(S') = I^\sharp(S,p) + I^\sharp(S,q) + PI^\sharp(S).\]
Here and throughout we assume that $\omega = \emptyset$ when it is not denoted.

\item (Lemma 4.2 of \cite{km_iabnh}) Let $(S,\omega)$ be a cobordism between $(L_1,\omega_1)$ and $(L_2, \omega_2)$. Let $R_+$ and $R_-$ be the two standard embedded copies of $\rr P^2$ in $S^4$ with self intersection $+2$ and $-2$ respectively. Let $\pi$ be a disk whose boundary is the generator of $H_1(\rr P^2)$.
Then
\[I^\sharp(S \# R_+)_{\omega + \pi} = I^\sharp(S)_\omega\]
and
\[I^\sharp(S \# R_-)_{\omega + \pi} = PI^\sharp(S)_\omega.\]

\item (Section 5.5 of \cite{km_unknot}, Section 2.2 of \cite{km_adoihfw}, Section 5.3 of \cite{km_iabnh}, ``Kunneth formula for split links'') Let $L$ be a split link, so that $L = L_0 \smallcoprod L_1$, and $L_0$ and $L_1$ are contained in disjoint balls in $S^3$. Then
\[I^\sharp(L) \simeq I^\sharp(L_0) \otimes I^\sharp(L_1),\]
and this is natural with respect to cobordisms with dots.

This is shown in Section 5.5 of \cite{km_unknot} using a version of excision without local coefficients, Hopf link instead of a $\theta$ web, and without dots. There is an argument in Section 2.2 of \cite{km_adoihfw} for why it does not matter whether one uses a $\theta$ web or a Hopf link, and it is explained in Section 5.3 of \cite{km_iabnh} why it still works with local coefficients. The proof of functoriality in \cite{km_unknot} carries over with no problems to the situation of cobordisms with dots.

\item (Section 5.2 of \cite{km_iabnh}) Let $U_\ell$ be the $\ell$ component unlink. Then $I^\sharp(U_0)$ is a free module of rank $1$ over $\cS$, which we write as $I^\sharp(U_0) = \cS u_0$, and $I^\sharp(U_1)$ is the free module over $\cS$ of rank $2$, which we write as $I^\sharp(U_1) = \cS u_+ \oplus \cS u_-$. For $D$ the standard disk viewed as a cobordism from the empty link to the unknot, and $q$ a point with orientation compatible with the choice of orientation of the knot, 
\[I^\sharp(D)(u_0) = u_+ \text{ and } I^\sharp(D, q)(u_0) = u_-.\]

Moreover, if $D_o$ is the standard disk viewed as a cobordism from the unknot to the empty link, and $q$ a point with orientation compatible with the choice of orientation of the knot,
\[I^\sharp(D_o)(u_-) = 1 \text{, } I^\sharp(D_o)(u_+) = 0 \text{, } I^\sharp(D_o,q)(u_+) = 1 \text{, } I^\sharp(D_o,q)(u_-) = P.\]

For $U_\ell$, by the previous point, we have
\[I^\sharp(U_\ell) = (\cS u_+ \oplus \cS u_-)^{\otimes \ell}.\]

\item (Section 5.4 of \cite{km_iabnh}) Let $m$ and $\Delta$ denote the standard ``pair of pants'' cobordisms between the two component unlink $U_2$ and the unknot $U_1$, the merge,
\[m:U_2 \to U_1\]
and the split
\[\Delta:U_1 \to U_2.\]
The map on $I^\sharp$ induced by $m$ (with no dots) is given by
\begin{align} \label{fla:merge}
\begin{split}
    u_{+} \otimes u_{+}     &\mapsto u_{+} \\
    u_{\pm} \otimes u_{\mp} &\mapsto u_{-} \\
    u_{-} \otimes u_{-}     &\mapsto P u_{-} + Q u_{+},
\end{split}
\end{align}
and the map induced by $\Delta$ (with no dots) is given by
\begin{align} \label{fla:split}
\begin{split}
    u_{+}  &\mapsto u_{+}\otimes u_{-} + u_{-}\otimes u_{+} +
   Pu_{+}\otimes u_{+}\\
   u_{-}  &\mapsto u_{-}\otimes u_{-} + Q u_{+} \otimes u_{+}.
\end{split}
\end{align}

\end{enumerate}

%%%% Sept 20 change 4
\section{A technical result for instantons}
\label{sec:techI}

In this section we prove the main technical result for instanton Floer homology $I^\sharp$ which we will use to prove Theorem \ref{thm:main}.
To do so, we will use a classical result in Morse theory, Lemma \ref{lem:break} below.
We state it in the most convenient form for us, and give a quick sketch its proof.

%\begin{theorem}
%Let $K_1$ and $K_2$ be knots in $S^3$. Suppose that there is an unorientable knot cobordism $\Sigma$ in $I \times S^3$ from $K_1$ to $K_2$ with $M$ local maxima. Then
%\begin{equation}
%\label{eq:I_main}
%\OrdI(K_1) \leq \max\set{\OrdI(K_2), M} + \gamma(\Sigma)
%%\OrdI(K_1) \leq \max\set{\OrdI(K_2), M} - \chi(\Sigma)
%\end{equation}
%\label{thm:main_instantons}
%\end{theorem}

\begin{definition}
Given a knot $K$ in $S^3$ and a band $B$ for $K$, i.e.~an embedded rectangle $B$ in $S^3$ which intersects $K$ in two opposite sides, we say that $B$ is \emph{orientable with respect to $K$} if the knot $K$ and the result of band surgery on $K$ along $B$ can be given coherent orientations (equivalently, if surgering $K$ along $B$ gives a 2-component link).
\end{definition}

%\MM{We should decide a coherent notation, possibly sensible too. Any strong opinions? Maybe the best one is $K=K_0$ and $K'=K_6$ for the extrema, and $K_i$ or $L_i$ after step $i$ (was this your previous notation?). }
\begin{lemma}
\label{lem:break}
Let $\Sigma \subset I \times S^3$ be a non-orientable cobordism between knots $K$ and $K'$ with $m$ local minima, $b$ saddles, and $M$ local maxima. Then, after an isotopy rel boundary, we can break it into a sequence of cobordisms as follows.
\begin{enumerate}
\item \label{it:B} %Births
$m$ births (from $K_1=K$ to $L_1$);
\item \label{it:BS} %Birth-cancelling Saddles
$m$ band surgeries that join the various components of the link (from $L_1$ to $K_1'$);
\item \label{it:OS} %Orientable Saddles
$b-(m+M+1)$ band surgeries orientable with respect to $K_1'$ (this cobordism ends with a knot or a 2-component link $L'$);
\item \label{it:US} %Unorientable Saddle
$1$ band surgery unorientable with respect to $K_1'$ (this cobordism goes from $L'$ to a knot $K_2'$);
\item \label{it:DS} %Death-cancelling Saddles
$M$ band surgeries that split the knot $K_2'$ into $M+1$ components;
\item \label{it:D} %Deaths
$M$ deaths.
\end{enumerate}

Moreover, in this decomposition, the attaching arcs of the $b$ bands on $K_1'$ can be assumed to be all disjoint, and we can assume that both attaching arcs of the unorientable band are already contained in $K_1$.

\end{lemma}

\begin{proof}[Sketch of the proof]

We can arrange all births to appear first and all deaths to appear last (steps \eqref{it:B} and \eqref{it:D}). We can also find bands that connect the various components (steps \eqref{it:BS} and \eqref{it:DS}). Thus, we can restrict to the part of the cobordism between $K_1'$ and $K_2'$, which consists of saddles (i.e., band surgeries). Note that both $K_1'$ and $K_2'$ are knots.

If all bands are orientable with respect to $K_1'$, then all $\Sigma$ would be orientable, so there is at least one band unorientable with respect to $K_1'$.

Arrange for all bands from $K_1'$ to $K_2'$ to appear all at the same time.

If there is more than one band unorientable with respect to $K_1'$, pick one of them (call it $B$) and slide it following the surgery of $K_1'$ along all the other bands. When $B$ slides over an orientable band, it stays unorientable. When $B$ slides over an unorientable band, it becomes orientable. Note that eventually it must slide over an unorientable band because $K_2'$ is connected, so $K_2'\sm B$ consists of just two arcs.

Repeat until you have only one unorientable band left.

%%%%%%%%%%% Sept 20 continue here with change 4!!!
If $B$ is the unique unorientable band, then you can slide its endpoints along $L'$ so that they are disjoint from all the other (oriented) bands, so we can think of it as in $K_1$.
\end{proof}

The main technical result of this section, needed to prove Theorem \ref{thm:main}, is the following Proposition.

\begin{propn}
\label{propn:main_technical}
\label{prop:I}
Let $S$ be a cobordism from $K$ to $K'$ with $m$ local minima, $b$ saddles, and $M$ local maxima. Then there is a surface $\omega$ that meets $S$ cleanly and only at $\partial \omega \subset S$, whose boundary is a circle in $S$ such that for $\overline{\omega}$ its mirror, we have
\begin{equation}
\label{eq:I}
P^M I^\sharp(\overline{S} \circ S)_{\omega \cup \overline{\omega}} = P^{b-m}\Id.
\end{equation}
\end{propn}

Towards this goal, let us start by doing some computations of maps induced by cobordisms with $\omega$.

First let us understand the dependence of $I^\sharp(\Sigma)_\omega$ on $\omega$, when $\omega$ is a surface with boundary on $\Sigma$, which intersects $\Sigma$ cleanly and only at $\partial \omega \subset S$. Note that  for a link $L$ in $S^3$, up to isomorphism, $I^\sharp(S^3, L)_\omega$ depends only on the homology of $[\partial \omega] \in H_0(L; \zz/2)$, because it counts flat connections and instantons on spaces determined by the homology class. 

Similarly, $I^\sharp(\Sigma)_\omega$ depends only on the homology class $[\partial \omega] \in H_1(\Sigma, \zz/2)$. This is because the map counts instantons on a moduli space built from $\Sigma$, $[\partial \omega] \in H^1(\Sigma, \zz/2)$ and $[\omega, \partial\omega] \in H_2(X,\Sigma, \zz/2)$, and $H_1(\Sigma) \simeq H_2(X, \Sigma)$, for $X = S^3 \times \rr$.

%%% (SAY SOMETHING ABOUT WHY $\omega$ IS NOT JUST DEFINED IN TERMS OF HLGY!!!)

From here, we can see that for a cylinder $\Sigma$ and $\omega$ given by either a small disk or a small tube with boundary on $\Sigma$, as in Figure \ref{fig:cylinders_with_omega}, $I^\sharp(\Sigma)_\omega$ induces the identity: here $\partial \omega$ is trivial in $H_1(\Sigma)$, and $[\omega,\partial \omega]$ is also trivial in $H_2(X, \Sigma,\zz/2)$. 

\begin{figure}[ht!]
\centering
\includegraphics{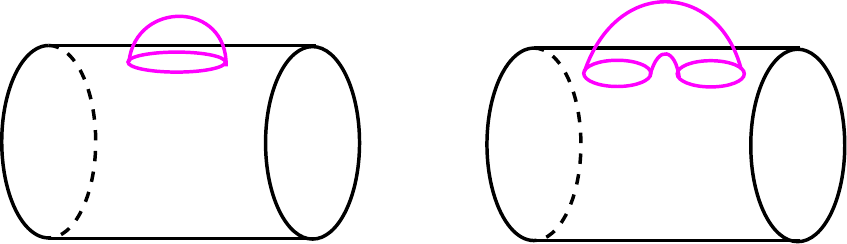}
\caption{\label{fig:cylinders_with_omega} Cylinders with the magenta surfaces depicting $\omega$.}
\end{figure}

When the cobordism in Figure \ref{fig:null_to_arc} is composed with its inverse, the map induced is the identity. Moreover, up to isomorphism $I^\sharp(U, \omega)$ depends only on $[\partial \omega] \in H_0(U; \zz/2)$, so the two ends of the cobordism have the same instanton Floer homology. Thus, the cobordism in Figure \ref{fig:null_to_arc} induces an isomorphism. 
\begin{figure}[ht!]
\centering
\includegraphics{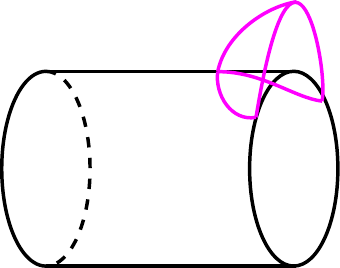}
\caption{\label{fig:null_to_arc}}
\end{figure}

We will call the two generators of the instanton Floer homology of the unknot with an arc $\omega$ on the right, which is depicted in Figure \ref{fig:unknot_with_arc}, $x_+$ and $x_-$, so that in the $u_\pm$ and $x_\pm$ bases, the cobordism depicted in Figure \ref{fig:null_to_arc} is the identity matrix.

\begin{figure}[ht!]
\centering
\includegraphics{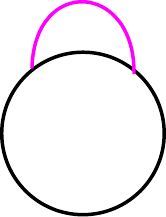}
\caption{\label{fig:unknot_with_arc}}
\end{figure}

The cobordism from the two component unlink to itself induced by two standard cylinders with $\omega$ as a tube between them, as depicted in Figure \ref{fig:cyls_omega_between}, induces the identity map, because in this situation, $(\omega, \partial\omega)$ is trivial in homology in $(S^3 \times I, \Sigma)$.
\begin{figure}[ht!]
\centering
\includegraphics{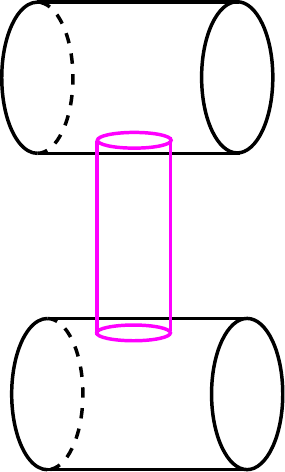}
\caption{\label{fig:cyls_omega_between}}
\end{figure}

The same is true for the map depicted in Figure \ref{fig:cyls_pants_omega} precomposed with its mirror. Thus, the map induced by the cobordism depicted in Figure \ref{fig:cyls_pants_omega} is an isomorphism whose inverse is its mirror image. Here, we are identifying the link with $\omega$ on the right end of Figure \ref{fig:cyls_pants_omega} with the unlink with empty $\omega$ via the isomorphism induced by Figure \ref{fig:null_to_arc}, and the link with $\omega$ on the left has isomorphic instanton Floer homology. 

\begin{figure}[ht!]
\centering
\includegraphics{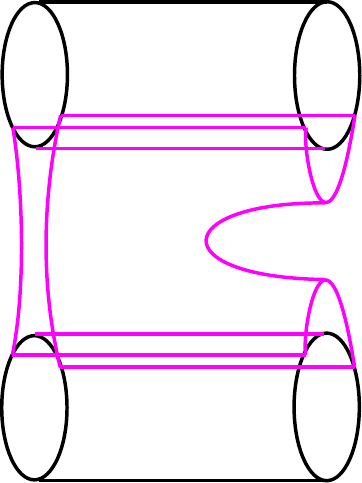}
\caption{\label{fig:cyls_pants_omega}}
\end{figure}

For the link on the left in Figure \ref{fig:cyls_pants_omega}, its homology is then a free module of rank $4$ over $\cS$. Let $x_{++}, x_{+-}, x_{-+}, x_{--}$ be a basis of this homology so that if we choose the $x_+ \otimes x_+, x_+ \otimes x_-, x_- \otimes x_+, x_- \otimes x_-$ basis for the two component unlink on the right, the matrix the cobodism induces is the identity. (Recall that $x_\pm$ are the basis elements of the instanton homology of the unknot with an arc, so that the cobordism of Figure \ref{fig:null_to_arc} induces the identity matrix.

%\MM{Use subsections?}
A central step in our proof will be dealing with a cobordism that flips an unknot, but does not change $\omega$. To describe this, consider a link $L$ with decoration $\omega$, which has an unknot component $U$ that is split from the rest of $L$; we may isotope $U$ so that it is a geometric circle. Suppose that $\omega$ has two endpoints on $U$, $p$ and $q$, which we may isotope to be the endpoints of a diameter of $U$. Then the flip cobordism is a cobordism in $I \times S^3$ that is traced by the isotopy obtained by rotating $U$ by $\pi$ about the diameter $pq$. So this is an isotopy that does not change $\omega$ and reverses the orientation of one of the two components.

%%%%%%%%%% Sept 21 change 5

\begin{claim}
The map on the instanton homology of $U_2$ with $\omega$ consisting of two arcs, each going between the two components that results from flipping one of the unknots (as described above) in a way that does not change $\omega$ is the identity map. 
\label{claim:omega_between_flip}
\begin{figure}[ht!]
\centering
\includegraphics{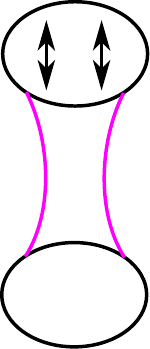}
\caption{\label{fig:omega_between_flip}}
\end{figure}

\end{claim}

\begin{proof}
By composing with the isomorphisms induced by the cobordism depicted in Figure \ref{fig:cyls_pants_omega}, if $\Phi$ is the matrix associated to the flip in the $x_{++}, x_{+-}, x_{-+}, x_{--}$ basis, then 
\[\Phi = {\left[
\begin{array}{cccc}
a & 0 & b & 0 \\
0 & a & 0 & b \\
c & 0 & d & 0 \\
0 & c & 0 & d \\
\end{array}
\right]}
\]
where $\Phi_1 = \mattwo{a}{b}{c}{d}$ is the flip on the unknot with an arc in the $x_+,x_-$ basis, depicted in Figure \ref{fig:unknot_flip}. This is because $\Phi$ is the matrix for $\Phi_1 \otimes \text{Id}$, for $\text{Id}$ the identity map, in the $x_+ \otimes x_+, x_+ \otimes x_-, x_- \otimes x_+, x_- \otimes x_-$ basis, and we are using the basis of the instanton homology of Figure \ref{fig:omega_between_flip} corresponding to this basis.

\begin{figure}[ht!]
\centering
\includegraphics{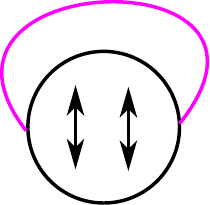}
\caption{\label{fig:unknot_flip}}
\end{figure}

Now let us compute some of the entries of $\Phi_1$. Note that if we pre-compose or post-compose $\Phi_1$ with caps like those in Figure \ref{fig:omega_caps}, we get back the cap itself. These caps induce the maps $\vecttwo{1}{0}$ and $[0,1]$, so from these compositions, we can see that $a = d = 1$ and $c= 0$. 
\begin{figure}[ht!]
\centering
\includegraphics{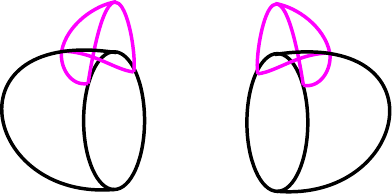}
\caption{\label{fig:omega_caps}}
\end{figure}

Note that if we did not have $\omega$, then we could do the same argument with a cap with a dot, and using the fact that doing a flip and then a cap with a dot is the same as doing a negative dot, we would be able to get the remaining entry, $b$, and recover Proposition 5.8 of \cite{km_iabnh}, in which the flip map does \textit{not} induce the identity. However, because we have $\omega$ here, this does not work: the flip changes which side of $\partial{\omega}$ the dot is on.

Going back to our computation, we now have that
\[\Phi = {\left[
\begin{array}{cccc}
1 & 0 & b & 0 \\
0 & 1 & 0 & b \\
0 & 0 & 1 & 0 \\
0 & 0 & 0 & 1 \\
\end{array}
\right]}.
\]

We would now like to show that $b$ is $0$. Consider the pair of pants cobordism with $\omega$ as two half-disks from the unlink with two arcs going between components to the unknot, as depicted in Figure \ref{fig:pants_omega_between}. Because we can precompose with isomorphisms to make a regular merge with a null-homotopic disk on top, as in Figure \ref{fig:pants_omega_disk}, we see that Figure \ref{fig:pants_omega_between} induces the same as the merge map, if we use the $x_{++}, x_{+-}, x_{-+}, x_{--}$ basis. Here we are using that the reverse of the map in Figure \ref{fig:cyls_pants_omega} is also the identity matrix with our choice of basis. 
\begin{figure}[ht!]
\centering
\includegraphics{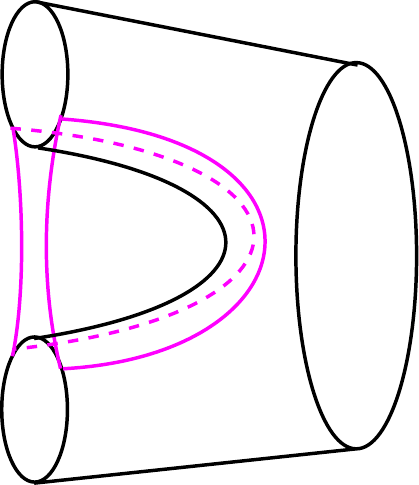}
\caption{\label{fig:pants_omega_between}}
\end{figure}

\begin{figure}[ht!]
\centering
\includegraphics{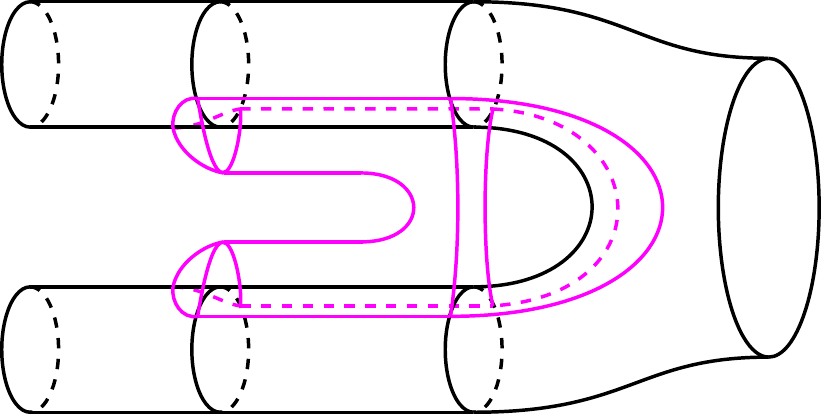}
\caption{\label{fig:pants_omega_disk}}
\end{figure}

Thus, in this basis, it induces the map
\[m = {\left[
\begin{array}{cccc}
1 & 0 & 0 & Q \\
0 & 1 & 1 & P \\
\end{array}
\right]}.\]

Similarly, the reverse of this cobordism induces the same map as $\Delta$, so it induces 
\[\Delta = {\left[
\begin{array}{cc}
P & Q \\
1 & 0 \\
1 & 0 \\
0 & 1 \\
\end{array}
\right]}.\]

Thus, composing $m \circ \Phi \circ \Delta$, we get the map
\[m \circ \Phi \circ \Delta = \mattwo{b+P}{0}{0}{b+P}.\]
However, if we compose these cobordisms, we get a Klein bottle, which is a connected sum of $\rr P^2_+$ and $\rr P^2_-$, with $\omega$ given by two disks, one on each $\rr P^2$, such that the boundary circle of the disk is the generator of $H_1(\rr P^2)$. It is shown in \cite{km_iabnh} that this Klein bottle with these $\omega$ induces the map $P \cdot \Id$, so $b = 0$, as desired.
\end{proof}

\begin{claim}
Let $S \subset S^3 \times [0,2]$ be a cobordism from $K_1$ to $K_2$ such that in $S^3 \times [0,1]$ it is the cylinder on $K_1$ and in $S^3 \times [1,2]$ it consists of adding a single non-orientable band. More precisely, we may consider a band $B \subset S^3$ with vertices $A_1A_2A_3A_4 \subset S^3$ with $A_1A_2$ and $A_3A_4$ on $K_1$, as in Figure \ref{fig:K1_with_band}. In $S^3 \times [1,2]$, $S$ then looks like $(K_1 \backslash (A_1A_2 \cup A_3A_4)) \times [1,2]$ away from the band $B \times [1,2]$, and within the band it goes from $A_1A_2 \cup A_3A_4$ at time $1$ to $A_2A_3 \cup A_4 A_1$ at time $2$.

The cobordism is depicted in frames in Figure \ref{fig:S_in_frames}. 

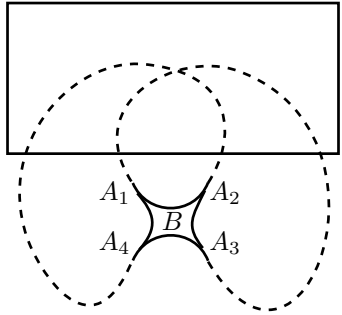
\begin{figure}[ht!]
%\centering
%\includegraphics[width=0.3\textwidth]{K1_with_band.png}
%% Creator: Inkscape 1.0.1 (c497b03c, 2020-09-10), www.inkscape.org
%% Accompanies image file 'K1_with_band.pdf' (pdf, eps, ps)
\begingroup%
  \makeatletter%
  \providecommand\color[2][]{%
    \errmessage{(Inkscape) Color is used for the text in Inkscape, but the package 'color.sty' is not loaded}%
    \renewcommand\color[2][]{}%
  }%
  \providecommand\transparent[1]{%
    \errmessage{(Inkscape) Transparency is used (non-zero) for the text in Inkscape, but the package 'transparent.sty' is not loaded}%
    \renewcommand\transparent[1]{}%
  }%
  \providecommand\rotatebox[2]{#2}%
  \newcommand*\fsize{\dimexpr\f@size pt\relax}%
  \newcommand*\lineheight[1]{\fontsize{\fsize}{#1\fsize}\selectfont}%
  \ifx\svgwidth\undefined%
    \setlength{\unitlength}{125.74286702bp}%
    \ifx\svgscale\undefined%
      \relax%
    \else%
      \setlength{\unitlength}{\unitlength * \real{\svgscale}}%
    \fi%
  \else%
    \setlength{\unitlength}{\svgwidth}%
  \fi%
  \global\let\svgwidth\undefined%
  \global\let\svgscale\undefined%
  \makeatother%
  \begin{picture}(1,0.92834691)%
    \lineheight{1}%
    \setlength\tabcolsep{0pt}%
    \put(0,0){\includegraphics[width=\unitlength,page=1]{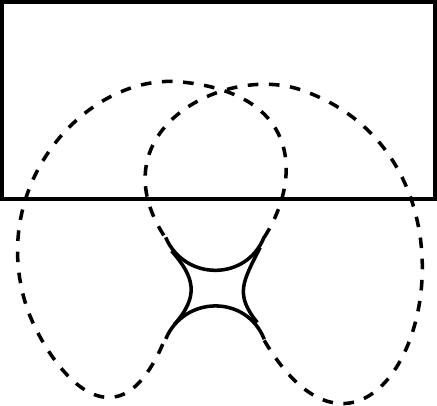}}%
    \put(0.60896444,0.33342103){\color[rgb]{0,0,0}\makebox(0,0)[lt]{\lineheight{1.25}\smash{\begin{tabular}[t]{l}\small $A_2$\end{tabular}}}}%
    \put(0.60896442,0.18303913){\color[rgb]{0,0,0}\makebox(0,0)[lt]{\lineheight{1.25}\smash{\begin{tabular}[t]{l}\small $A_3$\end{tabular}}}}%
    \put(0.27757942,0.33342101){\color[rgb]{0,0,0}\makebox(0,0)[lt]{\lineheight{1.25}\smash{\begin{tabular}[t]{l}\small $A_1$\end{tabular}}}}%
    \put(0.27757942,0.18303913){\color[rgb]{0,0,0}\makebox(0,0)[lt]{\lineheight{1.25}\smash{\begin{tabular}[t]{l}\small $A_4$\end{tabular}}}}%
    \put(0.46682341,0.24396559){\color[rgb]{0,0,0}\makebox(0,0)[lt]{\lineheight{1.25}\smash{\begin{tabular}[t]{l}\small $B$\end{tabular}}}}%
  \end{picture}%
\endgroup%

\caption{\label{fig:K1_with_band} The cobordism $S$ is cylindrical on the dotted part.}
\end{figure}

\begin{figure}[ht!]
\centering
\includegraphics{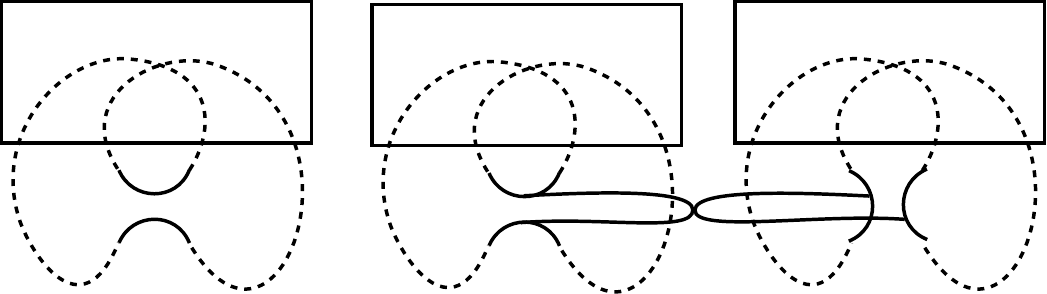}
\caption{\label{fig:S_in_frames}}
\end{figure}

Then there is a surface $\omega$ with boundary in the interior of $S$, such that $\omega$ meets $S$ only at the boundary, where they meet cleanly, such that for $\overline{S}$ the reverse of $S$, with corresponding $\overline{\omega}$, 
\[I^\sharp(\overline S)_{\overline{\omega}} \circ I^\sharp(S)_{\omega} = P \cdot \Id:I^\sharp(K_1) \to I^\sharp(K_1).\]
\label{claim:nonorientable_tube_omega}
\end{claim}

\begin{proof}

Observe that $A_1$ and $A_3$ split $K_1$ into two parts, which we call $a$ and $b$ (these are coloured magenta and blue respectively in Figure \ref{fig:K1_with_band_coloured_ab}. Let $c$ be the diagonal on the band that goes from $A_1$ to $A_3$. 

\begin{figure}[ht!]
\centering
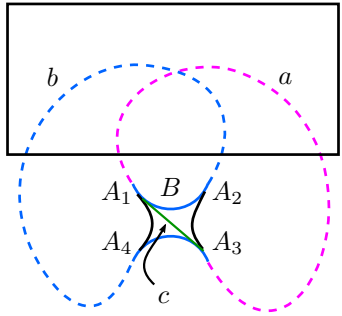
\caption{\label{fig:K1_with_band_coloured_ab} The cobordism $S$ is cylindrical on the dotted part.}
\end{figure}

Consider $a \cup c$ as a knot in $S^3$ and let $F_0$ be a Seifert surface of it. Then, $F_0$ is a surface with corners, with boundary $a \cup c$, and which meets $b$ at the ends, $A_1$ and $A_3$. We may isotope $a,b,c$ so that $F_0$ meets $b$ cleanly at the ends and transversely in the interior, as in Figure \ref{fig:F0_with_b_intersecting}

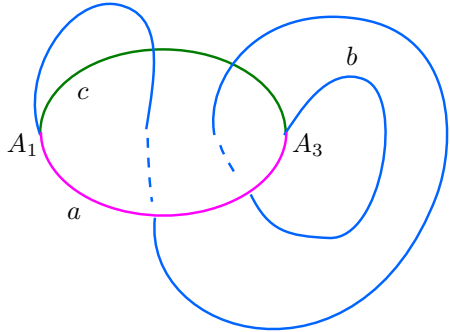
\begin{figure}[ht!]
\centering
%% Creator: Inkscape 1.0.1 (c497b03c, 2020-09-10), www.inkscape.org
%% Accompanies image file 'F0_with_b_intersecting.pdf' (pdf, eps, ps)
\begingroup%
  \makeatletter%
  \providecommand\color[2][]{%
    \errmessage{(Inkscape) Color is used for the text in Inkscape, but the package 'color.sty' is not loaded}%
    \renewcommand\color[2][]{}%
  }%
  \providecommand\transparent[1]{%
    \errmessage{(Inkscape) Transparency is used (non-zero) for the text in Inkscape, but the package 'transparent.sty' is not loaded}%
    \renewcommand\transparent[1]{}%
  }%
  \providecommand\rotatebox[2]{#2}%
  \newcommand*\fsize{\dimexpr\f@size pt\relax}%
  \newcommand*\lineheight[1]{\fontsize{\fsize}{#1\fsize}\selectfont}%
  \ifx\svgwidth\undefined%
    \setlength{\unitlength}{166.07206052bp}%
    \ifx\svgscale\undefined%
      \relax%
    \else%
      \setlength{\unitlength}{\unitlength * \real{\svgscale}}%
    \fi%
  \else%
    \setlength{\unitlength}{\svgwidth}%
  \fi%
  \global\let\svgwidth\undefined%
  \global\let\svgscale\undefined%
  \makeatother%
  \begin{picture}(1,0.74410225)%
    \lineheight{1}%
    \setlength\tabcolsep{0pt}%
    \put(0,0){\includegraphics[width=\unitlength,page=1]{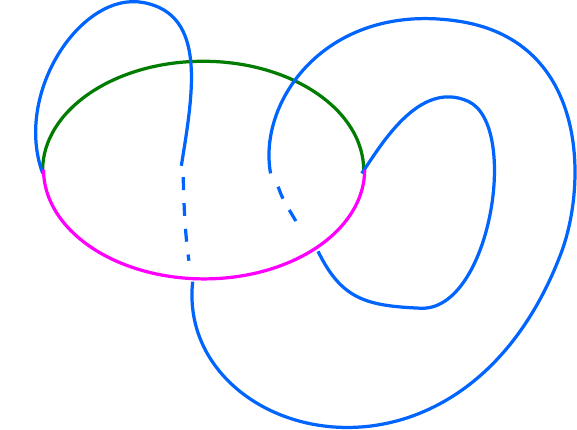}}%
    \put(-0.00151447,0.40572408){\color[rgb]{0,0,0}\makebox(0,0)[lt]{\lineheight{1.25}\smash{\begin{tabular}[t]{l}\small $A_1$\end{tabular}}}}%
    \put(0.64582718,0.40572407){\color[rgb]{0,0,0}\makebox(0,0)[lt]{\lineheight{1.25}\smash{\begin{tabular}[t]{l}\small $A_3$\end{tabular}}}}%
    \put(0.13518086,0.25306449){\color[rgb]{0,0,0}\makebox(0,0)[lt]{\lineheight{1.25}\smash{\begin{tabular}[t]{l}\small $a$\end{tabular}}}}%
    \put(0.76806835,0.60133308){\color[rgb]{0,0,0}\makebox(0,0)[lt]{\lineheight{1.25}\smash{\begin{tabular}[t]{l}\small $b$\end{tabular}}}}%
    \put(0.1573829,0.52584623){\color[rgb]{0,0,0}\makebox(0,0)[lt]{\lineheight{1.25}\smash{\begin{tabular}[t]{l}\small $c$\end{tabular}}}}%
  \end{picture}%
\endgroup%

\caption{\label{fig:F0_with_b_intersecting} Here $F_0$ is depicted as a disk though it could have higher genus.}
\end{figure}

If we choose an orientation of $F_0$ and $b$, then the intersection points may have positive or negative sign. We can increase the number of positive or negative intersections points without changing the isotopy type of the embedding of $K_1 \cup c$ into $S^3$ by twisting $b$ around $A_1$ or $A_3$, as in Figure \ref{fig:twisting_b}. Let us do this, adding either positive or negative intersections as needed until there are the same number of positive as negative interior intersection points between $b$ and $F_0$.

\begin{figure}[ht!]
\centering
%% Creator: Inkscape 1.0.1 (c497b03c, 2020-09-10), www.inkscape.org
%% Accompanies image file 'twisting_b.pdf' (pdf, eps, ps)
\begingroup%
  \makeatletter%
  \providecommand\color[2][]{%
    \errmessage{(Inkscape) Color is used for the text in Inkscape, but the package 'color.sty' is not loaded}%
    \renewcommand\color[2][]{}%
  }%
  \providecommand\transparent[1]{%
    \errmessage{(Inkscape) Transparency is used (non-zero) for the text in Inkscape, but the package 'transparent.sty' is not loaded}%
    \renewcommand\transparent[1]{}%
  }%
  \providecommand\rotatebox[2]{#2}%
  \newcommand*\fsize{\dimexpr\f@size pt\relax}%
  \newcommand*\lineheight[1]{\fontsize{\fsize}{#1\fsize}\selectfont}%
  \ifx\svgwidth\undefined%
    \setlength{\unitlength}{188.19528322bp}%
    \ifx\svgscale\undefined%
      \relax%
    \else%
      \setlength{\unitlength}{\unitlength * \real{\svgscale}}%
    \fi%
  \else%
    \setlength{\unitlength}{\svgwidth}%
  \fi%
  \global\let\svgwidth\undefined%
  \global\let\svgscale\undefined%
  \makeatother%
  \begin{picture}(1,0.44240816)%
    \lineheight{1}%
    \setlength\tabcolsep{0pt}%
    \put(0,0){\includegraphics[width=\unitlength,page=1]{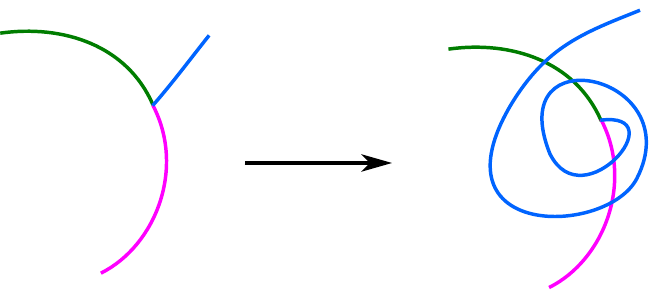}}%
    \put(0.23539808,0.06440619){\color[rgb]{0,0,0}\makebox(0,0)[lt]{\lineheight{1.25}\smash{\begin{tabular}[t]{l}\small $a$\end{tabular}}}}%
    \put(0.91248172,0.02391104){\color[rgb]{0,0,0}\makebox(0,0)[lt]{\lineheight{1.25}\smash{\begin{tabular}[t]{l}\small $a$\end{tabular}}}}%
    \put(0.98141871,0.28981118){\color[rgb]{0,0,0}\makebox(0,0)[lt]{\lineheight{1.25}\smash{\begin{tabular}[t]{l}\small $b$\end{tabular}}}}%
    \put(0.2878722,0.30044718){\color[rgb]{0,0,0}\makebox(0,0)[lt]{\lineheight{1.25}\smash{\begin{tabular}[t]{l}\small $b$\end{tabular}}}}%
    \put(0.09658434,0.40532067){\color[rgb]{0,0,0}\makebox(0,0)[lt]{\lineheight{1.25}\smash{\begin{tabular}[t]{l}\small $c$\end{tabular}}}}%
    \put(0.72543816,0.38936664){\color[rgb]{0,0,0}\makebox(0,0)[lt]{\lineheight{1.25}\smash{\begin{tabular}[t]{l}\small $c$\end{tabular}}}}%
  \end{picture}%
\endgroup%

\caption{\label{fig:twisting_b}}
\end{figure}

Now say that the intersection points are $A_1, p_1,p_2,\ldots p_{2k},A_3$, in order along $b$. Then if $p_i$ and $p_{i+1}$ are intersection points with opposite sign, we may remove a small disk around each of $p_i$ and $p_{i+1}$ and replace it with a small tube around the part of $b$ that goes from $p_i$ to $p_{i+1}$, thus reducing the number of intersection points. We may continue in this manner, removing adjacent, opposite-sign intersection points until none remain.

We now have a surface, which we call $F_1$ with boundary $a \cup c$, which intersects $b$ only at $A_1$ and $A_3$, where the intersection is clean. 

We now consider a surface $F_2 \subset S^3 \times [0,2]$ with boundary on $S$ which is given by the union of $F_1 \subset S^3 \times \set1$ with a disk sitting between $c \times 1 \subset S^3 \times [1,2]$ and $S$, as in Figure \ref{fig:S_omega_shaded}.

\begin{figure}[ht!]
\centering
\scalebox{0.9}{
\includegraphics{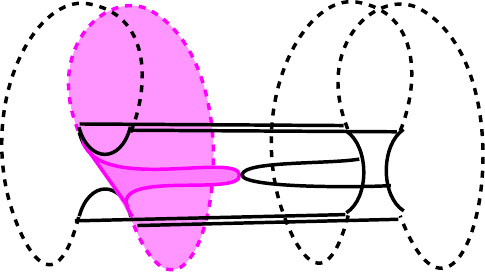}
}
\caption{\label{fig:S_omega_shaded}}
\end{figure}

Then this $F_2$ can have its corners smoothed out to a surface with boundary $\omega$.

Let us now show that for this $\omega$, we have
\[I^\sharp(\overline S)_{\overline{\omega}} \circ I^\sharp(S)_{\omega} = P \cdot \Id:I^\sharp(K_1) \to I^\sharp(K_1).\]

%\MM{New bit starts here.}
Let $\Sigma$ denote the composition of $S$ with $\overline{S}$, and let $\omega_{\Sigma} = \omega \cup \overline{\omega}$ be the decoration on this cobordism.
See Figure \ref{fig:nonorientable_tube_omega} for an illustration.

Let $\gamma$ denote the circle composed of the co-core of the band and its mirror, depicted in blue in Figure \ref{fig:nonorientable_tube_omega}. A regular neighbourhood of $\gamma$ in $\Sigma$ is a tube, represented in Figure \ref{fig:c2}. If we cut the surface along $\gamma$, we get the twice punctured cylinder as a cobordism from $K_1$ to itself.

\begin{figure}[ht!]
\centering
\scalebox{0.9}{
\includegraphics{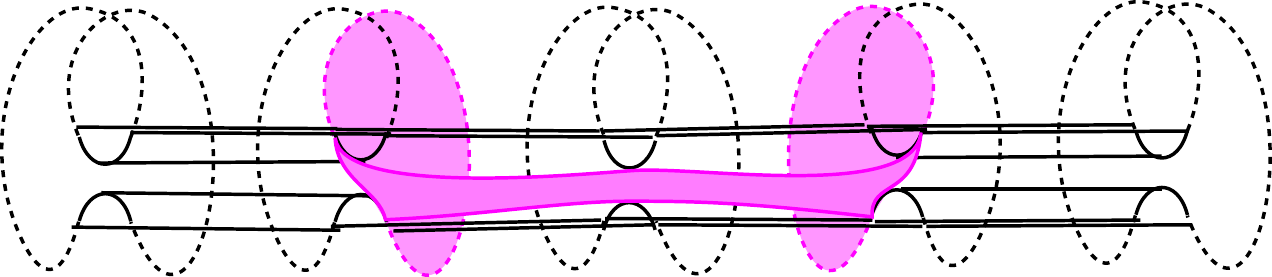}
}
\caption{\label{fig:nonorientable_tube_omega} This is $S$ with $\omega$ composed with the reverses $\overline{S}$ with $\overline{\omega}$.}
\end{figure}

Figure \ref{fig:c2} shows $\de\omega_{\Sigma}$ as well. The mod 2 homology class $[\de\omega_{\Sigma}]$ on the surface $\Sigma$ is the same as $[\gamma]$. One way to see it is to perform surgery on $\de\omega_{\Sigma}$ along the green arc in Figure \ref{fig:c2}: this operation does not change the homology class and it yields a curve which is easily checked to be isotopic to $\gamma$ in $\Sigma$.

\begin{figure}[ht!]
\centering
\includegraphics[width=\textwidth]{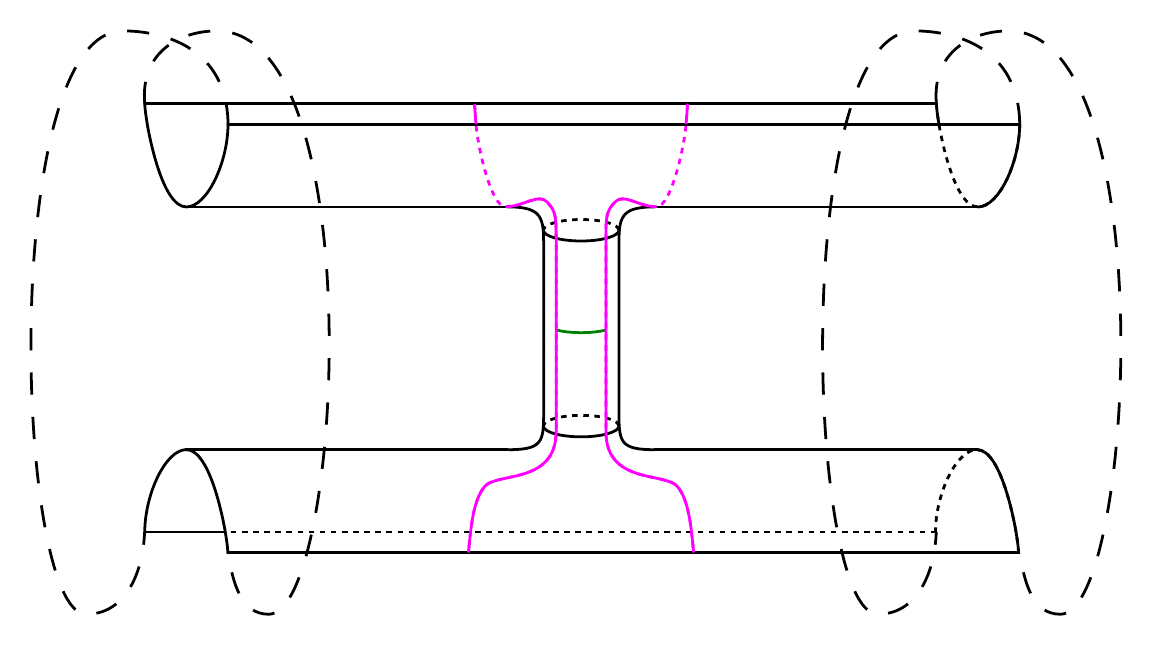}
\caption{\label{fig:c2}}
\end{figure}

Let $\Sigma'$ be the cobordism obtained from $\Sigma$ by inserting a flip in the tube in the centre of Figure \ref{fig:c2}, with the same decoration $\omega_\Sigma$.
Using Claim \ref{claim:omega_between_flip}, we will see below that $I^\sharp(\Sigma)_{\omega_{\Sigma}} = I^\sharp(\Sigma')_{\omega_{\Sigma}}$.
However, the curve $\de\omega_{\Sigma}$ is homologically trivial in $\Sigma'$. One can check it again by doing surgery on the green arc, but this time the extra flip ensured that the obtained curve is not $\gamma$, but a homotopically trivial one. Thus, $I^\sharp(\Sigma')_{\omega_{\Sigma}} = I^{\sharp}(\Sigma')_{\emptyset}$, since the map depends only on $[\de\omega_{\Sigma}]$.
If $\omega = \emptyset$ one can apply the neck cutting relation (Property \eqref{it:neckcutting} in Section \ref{sec:I_properties}) to obtain that
\[
I^\sharp(\Sigma') = P \cdot I^\sharp(I \times K_1) = P \cdot \Id_{I^\sharp(K_1)}.
\]

We still have to show that $I^\sharp(\Sigma)_{\omega_{\Sigma}} = I^\sharp(\Sigma')_{\omega_{\Sigma}}$.
To see this, isotope the tube in the middle as shown in Figure \ref{fig:c2bis}.

\begin{figure}[ht!]
\centering
\includegraphics[width=\textwidth]{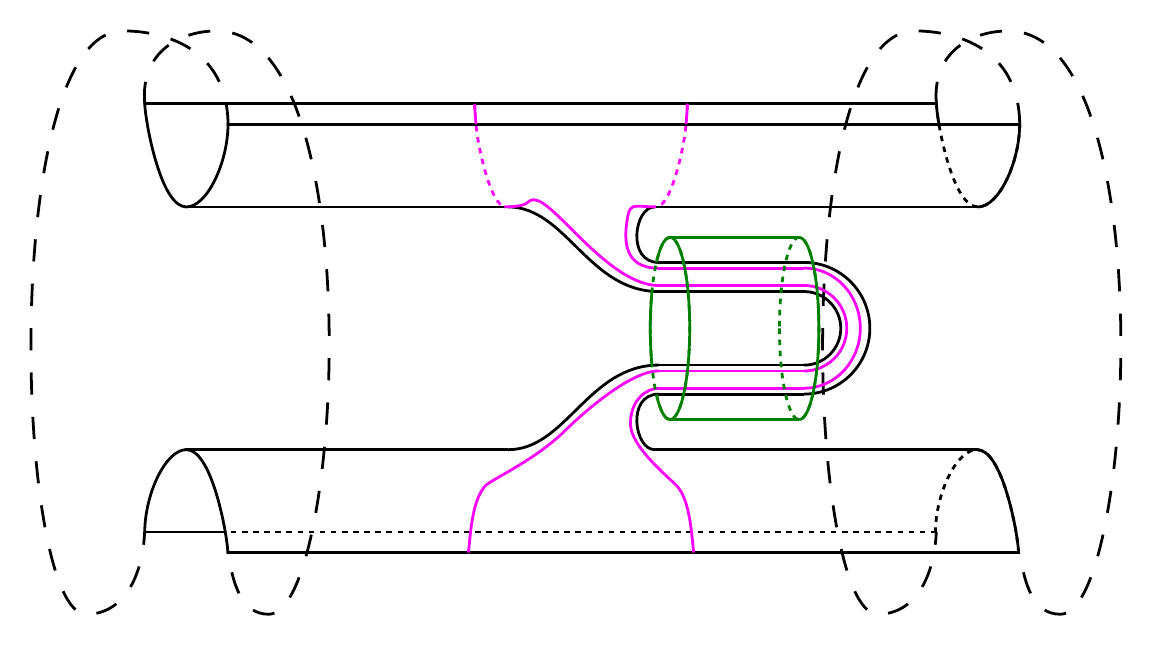}
\caption{\label{fig:c2bis}}
\end{figure}

Let's restrict our attention to the piece contained in the cylinder in green, which is the identity cobordism on a 2-component unlink.
By Claim \ref{claim:omega_between_flip} the map induced by this cobordism is the same that we get if we introduce a flip on one of the two components.
Since instanton Floer maps respect composition of cobordisms and disjoint unions, the map induced by the whole cobordism is not affected by the insertion of the flip, i.e.\ $I^\sharp(\Sigma)_{\omega_{\Sigma}} = I^\sharp(\Sigma')_{\omega_{\Sigma}}$.
\end{proof}
%\MM{End of the new bit.}

\begin{claim} Let $S$ be a cobordism from a knot $K_1$ to a knot $K_2$ such that $S$ consists of only $b$ bands. That is, there are no births nor deaths. Then there is a surface $\omega$ with boundary on $S$ such that
\[I^\sharp(\overline{S})_{\overline{\omega}} \circ I^\sharp(S)_\omega = P^b \cdot \Id.\]
\label{claim:bands_only}
\end{claim}

\begin{proof} We proceed by induction on $b$. The base case $b = 0$, is obvious. 

For the inductive step, we divide into two cases. If $S$ is orientable, then the statement holds for $\omega$ empty, because the cobordism $\overline{S} \circ S$ is the same as the cylinder on $K_1$ with $b$ orientable tubes, and the result follows from the tube cutting formula.

In the case that $S$ is not orientable, at least one of the bands of $S$ must be non-orientable with respect to $K$. In this case, let us write $S = S_r \circ S_u$, where $S_u$ is a cobordism consisting of the non-orientable band, and $S_r$ is the rest of the cobordism, which may or may not be orientable. 

Then by the induction hypothesis there is some $\omega_r$ such that
\[I^\sharp(\overline{S_r})_{\overline{\omega_r}} \circ I^\sharp(S_r)_{\omega_r} = P^{b-1} \cdot \Id\]

Applying Claim \ref{claim:nonorientable_tube_omega}, there is a surface $\omega_u$ with boundary on $S_u$ such that 
\[I^\sharp(\overline{S_u})_{\overline{\omega_u}} \circ I^\sharp(S_u)_{\omega_u} = P\cdot \Id.\]

The statement \[I^\sharp(\overline{S})_{\overline{\omega}} \circ I^\sharp(S)_\omega = P^b \cdot \Id\]
now follows.

\end{proof}

Now, we can proceed with the proof of Proposition \ref{propn:main_technical}. 

\begin{proof}[Proof of Proposition \ref{propn:main_technical}]
Applying Lemma \ref{lem:break}, we may break $S$ into pieces of the form.

\begin{enumerate}
\item %\label{it:B} %Births
$m$ births (from $K_1$ to $L_1$);
\item %\label{it:BS} %Birth-cancelling Saddles
$m$ band surgeries that joint the various components of the link (from $L_1$ to $K_1'$);
\item %\label{it:OS} %Saddles
$b-(m+M)$ band surgeries which may or may not be orientable, ending in a knot $K_2'$. 
\item %\label{it:DS} %Death-cancelling Saddles
$M$ band surgeries that split the knot $K_2'$ into $M+1$ components;
\item %\label{it:D} %Deaths
$M$ deaths.
\end{enumerate}

Let us call the cobordisms corresponding to the five steps $S_1,S_2, \ldots S_5$. We may isotope the cobordism in $S^3 \times \rr$ so that $S_i$ is in the $S^3 \times [i, i+1]$.

We will choose $\omega$ to be in $S^3 \times [2,3]$, so that its boundary is in $S_3$ as in Claim \ref{claim:bands_only} so that 
\[I^\sharp(\overline{S_3})_{\overline{\omega}} \circ I^\sharp(S_3)_\omega = P^{b-m-M} \cdot \Id.\]

The proof now proceeds the same way as the proof of Proposition 4.1 of \cite{JMZ}. The main argument is by considering a cobordims $\Sigma$ that comes from adding $M$ tubes connecting points on the death-caps to their mirrors. 

By Lemma 3.2 of \cite{km_iasciok}, for a connected, oriented cobordism $\Sigma$ if $\Sigma'$ is obtained from $\Sigma$ by adding a tube between points $p$ and $q$, then
\begin{align}
I^\sharp(\Sigma') & =I^\sharp(\Sigma, p) + I^\sharp(\Sigma, q) + P I^\sharp(\Sigma) =  P I^\sharp(\Sigma),
\label{eqn:tube_cutting}
\end{align}
where the second equality is because $\Sigma$ is connected, so $I^\sharp(\Sigma, p)$ and $I^\sharp(\Sigma, q)$ induce the same map, and since we are working over characteristic two, they cancel.

Let $\Sigma_1$ denote the cobordism coming from taking $\overline{S} \circ S$ and adding $M$ tubes, one for each death connecting a point in the death to its reverse, so that $\Sigma_1 = S_1S_2S_3S_4\overline{S_4}\overline{S_3}\overline{S_2}\overline{S_1}$. Applying Equation \eqref{eqn:tube_cutting} for each death, to the part of the cobordism from $K_3$ to itself coming from doing $S_4$, $S_5$ and their reverses, we see that 
\[I^\sharp(\Sigma_1) = P^M I^\sharp(\overline{S} \circ S).\]
Here, we are allowed to use the above result because $S_4S_5\overline{S_5}\overline{S_4}$ and $S_4 \overline{S_4}$ are both orientable and connected.

In $\Sigma_1 = (S_1S_2S_3S_4) \overline{(S_1S_2S_3S_4)}$, $M$ splitting bands of $S_4$ and their reverses, cap off the ends, and call the resulting cobordism $\Sigma_2$. Then for the same reason as the above, we have
\[I^\sharp(\Sigma_1) = P^{M} I^\sharp(\Sigma_2),\]
because again, $S_3S_4\overline{S_4}\overline{S_3}$ and $S_3\overline{S_3}$ are both orientable and connected. 

Now we have $\Sigma_2 = (S_1S_2S_3) \overline{(S_1S_2S_3)}$.

Because of our construction of $\omega$, $S_3 \overline{S_3}$ with $\omega \cup \overline{\omega}$	 falls under the setting of Claim \ref{claim:bands_only}, so the map it induces is $P^{b-m-M} \cdot \Id_{I^\sharp(K_2)}$. Thus, if we let $\Sigma_3 = (S_1S_2) \overline{(S_1S_2)}$, then 
\[I^\sharp(\Sigma_2)= P^{b-m-M} I^\sharp(\Sigma_3).\]

Now $\Sigma_3$ is given by a cylinder on $K_1$ and $m$ $S^2$s, with $m$ tubes, with the tubes connecting the $S^2$s and the cylinder in a tree-like fashion. Applying the tube-cutting formula, Lemma 3.2 of \cite{km_iasciok}, and observing that a sphere without any dots induces the $0$ map and the sphere with one dot induces the identity, we see that $I^\sharp(\Sigma_3)$ induces the same map as the cylinder, which is to say, the identity.

Putting all of this together, we get
\[ P^M I^\sharp(\overline{S} \circ S) = I^\sharp(\Sigma_1) = P^{b-m} \cdot \Id,\]
as desired.	
\end{proof}
	%
%
%
%
%Now we may proceed with the proof of Theorem \ref{thm:main_instantons}. 
%
%
%\begin{proof}[Proof of Theorem \ref{thm:main_instantons}] Let $b$ be the number of saddles of $\Sigma$ and $m$ be the number of births. 
%
%Using Proposition \ref{propn:main_technical} we see that there is $\omega$ such that 
%\[P^M I^\sharp(\overline{\Sigma} \circ \Sigma)_{\omega \cup \overline{\omega}} = P^{b-m} \Id.\]
%For any $x \in I^\sharp(K_1)$, since $I^\sharp(\Sigma) x \in I^\sharp(K_2)$, we have that $P^{\OrdI(K_2)}I^\sharp(\Sigma)x= 0$.
%
%Thus,
%\[I^\sharp(\overline{\Sigma})_{\overline{\omega}} \circ I^\sharp(\Sigma)_{\omega} P^{\OrdI(K_2)}x=0.\]
%
%This means
%\[P^MI^\sharp(\overline{\Sigma})_{\overline{\omega}} \circ I^\sharp(\Sigma)_{\omega} P^{\max\set{\OrdI(K_2)- M,0}}x=0,\]
%so 
%\[P^{b-m+\max\set{\OrdI(K_2)- M,0}}x=0,\]
%as desired.
%
%%%%%%%%%%%%%!!!!!!!!!!!!!!!!!!!!!!
%\end{proof}

\section{Background on unoriented knot Floer homology}
\label{sec:bgHFK}

Unoriented knot Floer homology was introduced by Oszv\'ath--Stipsicz--Szab\'o~\cite{Upsilon, tHFK}.
Fan~\cite{Fan} showed that an unorientable cobordism (with some extra data) induces maps on the unoriented knot Floer homology.
We now review the relevant definitions, following mostly \cite{ZFunctoriality} and \cite{Fan}.

\subsection{Zemke's oriented TQFT}

Cobordism maps in link Floer homology were first defined by \cite{JFunctoriality}. In this paper we use Zemke's setup~\cite{ZFunctoriality}, specified to unoriented link Floer homology in the case $Y = S^3$.

\begin{definition}%[]
\label{def:olink}
An \emph{oriented multi-based link} in $S^3$ is a triple $\olink = (L, \w, \z)$ consisting of an oriented, embedded link $L \subset S^3$, with two disjoint collections of basepoints $\w$ and $\z$ on $L$, such that each component of $L$ has at least two basepoints and the basepoints alternate between those in $\w$ and those in $\z$ as one traverses a component of $L$.
\end{definition}

To an oriented multi-based link $\olink$, Zemke's most general construction gives a curved $\F[U_\w, V_\z]$-complex $\CFLcm(\olink)$ up to $\F[U_\w, V_\z]$-equivariant chain homotopy. Here $\F[U_\w, V_\z]$ denotes the polynomial ring generated by a $U$ variable for each $\w$ basepoint and a $V$ variable for each $\z$ basepoint. The curved complex is also endowed with gradings and a filtration.
%\MM{Maybe we can get rid of this paragraph.}

In our case, we only need a simpler version of Zemke's complex, namely unoriented link Floer homology. This is defined as
\[
\CFLum(\olink) := \CFLcm(\olink) \otimes_{\F[U_\w, V_\z]} \F[U],
\]
where all variables act on $\F[U]$ as the multiplication by $U$.
For the reader familiar with Heegaard Floer homology, this is the free $\F[U]$-module generated by the intersection points $\T_{\ba} \cap \T_{\bb}$ in the symmetric product, with differential given by
\begin{equation}
\label{eq:CFK'differential}
\de \x = \sum_{\y \in \T_{\ba} \cap \T_{\bb}}	\sum_{\substack{\phi \in \pi_2(\x, \y)\\\mu(\phi)=1}} \#\RedModSpace(\phi) \cdot U^{n_{\o}(\phi)} \cdot \y,
\end{equation}
where $n_{\o}(\phi) = \sum_{w \in \w} n_w(\phi) + \sum_{z \in \z} n_z(\phi)$.

\begin{definition}
For a doubly-based knot $\oknot = (K, w, z)$ we also use the notations $\CFKum(\oknot)$ and $\HFKum(\oknot)$ for $\CFLum(\oknot)$ and $\HFLum(\oknot)$, respectively.
\end{definition}

If $\oknot_1 = (K, w_1, z_1)$ and $\oknot_2 = (K, w_2, z_2)$ are two doubly-based knot with the same underlying knot $K$, then $\HFKum(\oknot_1)$ and $\HFK(\oknot_2)$ are non-canonically isomorphic as $\F[U]$-modules. Thus, the following number is well-defined.

\begin{definition}
If $K$ is a knot, we define its \emph{unoriented torsion order} as
\[
\OrdU(K) = \min\set{n \geq 0 \,\middle|\, U^n \cdot \Tors = 0},
\]
where $\Tors$ is the torsion submodule of $\HFKum(\oknot)$, considered as a module over $\F[U]$. Here $\oknot$ is any doubly-based knot with underlying knot $K$.
\end{definition}

\begin{remark}
\label{rem:propertiesHFLum}
$\CFLum$ enjoys the following properties:
\begin{enumerate}
\item $\CFLum(\olink)$ is a genuine chain complex (i.e., its curvature vanishes), so one can compute its homology $\HFLum(\olink)$, known as the \emph{unoriented link Floer homology} of $\olink$. This is still an $\F[U]$-module.
\item For a doubly-based knot $\oknot = (K, w, z)$, $\HFKum(\oknot) \cong \F[U] \oplus \Tors$, where $\Tors$ is the torsion as an $\F[U]$-module.
\item For a doubly-based unknot $\ounknot = (U_1, w, z)$, $\HFKum(\ounknot) \cong \F[U]$.
\item \label{it:Kunneth}
Given doubly based knots $\oknot_1$ and $\oknot_2$,
\[
\CFKum(\oknot_1 \# \oknot_2) = \CFKum(\oknot_1) \otimes_{\F[U]} \CFKum(\oknot_2).
\]
As a consequence, for knots $K_1$ and $K_2$ in $S^3$,
\[
\OrdU(K_1 \# K_2) = \max\set{\OrdU(K_1), \OrdU(K_2)}.
\]
\item \label{it:mirror}
If $\m{\olink}$ is the mirror of $\olink$ (with same basepoints), then by \cite[Proposition 2.17]{tHFK}
\[
\CFLum(\m{\olink}) = \hom_{\F[U]}(\CFLum(\olink), \F[U]).
\]
As a consequence, for a knot $K$ in $S^3$,
\[
\OrdU(\m{K}) = \OrdU(K).
\]
\end{enumerate}
\end{remark}

%{\color{red}
%ADD PROPERTIES WHEN YOU NEED THEM
%}

%\begin{definition}
%An \emph{oriented surface with divides} is a pair $\dsurf = (\Sigma, \dec)$ such that
%\begin{enumerate}
%\item $\Sigma$ is a compact oriented surface;
%\item $\dec \subset \Sigma$ is a properly embedded 1-manifold;
%\item the components of $\Sigma \sm \dec$ are partitioned into two sub-surfaces, $\Sigma_\w$ and $\Sigma_\z$, which meet along $\dec$.
%\end{enumerate}
%\end{definition}

\begin{definition}
If $\olink_1 = (L_1, \w_1, \z_1)$ and $\olink_2 = (L_2, \w_2, \z_2)$ are two oriented multi-based links, an (\emph{oriented}) \emph{decorated link cobordism} from $\olink_1$ to $\olink_2$ is a pair $\dlcob = (\Sigma, \dec)$ such that
\begin{enumerate}
\item $\Sigma \subset I \times S^3$ is a properly embedded, compact, oriented surface with $\Sigma \cap \{0\} \times S^3 = \{0\} \times (-L_1)$ and $\Sigma \cap \{1\} \times S^3 = \{1\} \times L_2$;
\item $\dec \subset \Sigma$ is a properly embedded 1-manifold, which we we refer to as the \emph{decorations};
\item the components of $\Sigma \sm \dec$ are partitioned into two sub-surfaces, $\Sigma_\w$ and $\Sigma_\z$, which meet along $\dec$;
\item each component of $L_i \sm \dec$ contains exactly one basepoint of $\w_i \sqcup \z_i$;
\item $\w_1 \sqcup \w_2 \subset \Sigma_\w$ and $\z_1 \sqcup \z_2 \subset \Sigma_\z$.
\end{enumerate}
\end{definition}

\begin{definition}
\label{def:idcob}
The identity (decorated link) cobordism $\id_{\olink}$ from $\olink = (L, \w, \z)$ to itself is given by the surface $\Sigma = I \times L$ with decorations $\dec = I \times Q$, where $Q \subset L \sm (\w \cup \z)$ is a finite set such that the inclusion induces an isomorphism in $\pi_0$.
\end{definition}

By the work of Zemke~\cite{ZFunctoriality}, an oriented decorated link cobordism $\dlcob$ from $\olink_1$ to $\olink_2$ induces an $\F[U]$-equivariant map
\[
\Zmap_{\dlcob} \colon \HFLum(\olink_1) \to \HFLum(\olink_2).
\]

\begin{remark}
\label{rem:propertiesZ}
The map $\Zmap$ enjoys the following properties:
\begin{enumerate}
\item $\Zmap_{\dlcob}$ is invariant under isotopy of $\Sigma$ in $I \times S^3$ while fixing the boundary, and under isotopy of $\dec$ in $\Sigma$ while keeping
\[\de\dec \subset (L_1 \sm (\w_1 \cup \z_1))\cup(L_2 \sm (\w_2 \cup \z_2)).\]
\item If $\id_{\olink}$ is the identity cobordism from $\olink$ to itself, then
\[
\Zmap_{\id_{\olink}} = \id_{\HFLum(\olink)}.
\]
\item If $\dlcob_1$ and $\dlcob_2$ are oriented decorated link cobordisms from $\olink_1$ to $\olink_2$ and from $\olink_2$ to $\olink_3$ respectively, then one can stack $\dlcob_2$ on top of $\dlcob_1$ (after isotoping the decorations so that they match on the $\olink_2$ level), and obtain a new oriented decorated link cobordism $\dlcob_2 \circ \dlcob_1$ from $\olink_1$ to $\olink_3$. In such a case,
\[
\Zmap_{\dlcob_2 \circ \dlcob_1} = \Zmap_{\dlcob_2} \circ \Zmap_{\dlcob_1}.
\]
\item \label{it:compression}
If $\dlcob' = (\Sigma', \dec')$ is obtained from $\dlcob = (\Sigma, \dec)$ by attaching a tube with both feet in $\Sigma_\z$ (or both feet in $\Sigma_\w$), then $\Zmap_{\dlcob'} = U \cdot \Zmap_{\dlcob}$.
\end{enumerate}
\end{remark}

%{\color{red}
%ADD PROPERTIES WHEN NEEDED.
%}

\subsection{Fan's unoriented TQFT}

By the work of Fan~\cite{Fan}, the link Floer TQFT can be extended to the non-orientable case. We review the relevant definitions.

\begin{definition}
\label{def:dlink}
A \emph{disoriented link} in $S^3$ is a tuple $\dlink = (L, \p, \q)$ consisting of an unoriented, embedded link $L \subset S^3$, with two disjoint collections of points $\p$ and $\q$ on $L$, called the \emph{dividing set}, such that each component of $L$ has at least two points in the dividing set and the points in the dividing set alternate between those in $\p$ and those in $\q$ as one traverses a component of $L$.

Each component of $L \sm (\p \cup \q)$ is given a canonical orientation from $\q$ to $\p$.
We denote the oriented manifold $L \sm (\p \cup \q)$ by $\l$.
Note that these orientations do not glue to an orientation of $L$.
\end{definition}

As it is customary, we consider isotopic disoriented knots as different disoriented knots. It is well known that isotopies can induce non-trivial maps in knot Floer homology, such as the moving basepoint maps~\cite{Sarkar, ZQuasi}.

Definition \ref{def:dlink} looks exactly the same as Definition \ref{def:olink}, except that the link is now unoriented. However, we emphasise that the basepoints $\w \cup \z$ from \ref{def:olink} are ontologically different from the dividing set from \ref{def:dlink}. From a Morse theoretical viewpoint, the former arise as the intersection between $L$ and the middle level surface of a Morse function, whereas the latter are the index-$0$ and index-$3$ critical points of the function.

However, we can define a notion of compatibility between oriented decorated links and disoriented links.

\begin{definition}
We say that an oriented decorated link $\olink = (L, \w, \z)$ and a disoriented link $\dlink = (L, \p, \q)$ are \emph{compatible} if:
\begin{itemize}
\item the underlying unoriented link $L$ is the same (but note that in $\olink$ it also comes with an orientation);
\item $\p\cup\q$ is disjoint from $\w\cup\z$;
\item each component of $L \sm (\p\cup\q)$ contains exactly one basepoint in $\w \cup \z$;
\item the components of $L\sm(\w\cup\z)$ containing the $\p$ point are oriented from $z$ to $w$ (with orientation induced from $\olink$).
\end{itemize}
\end{definition}

\begin{remark}
\label{rem:linkcompatibility}
Every disoriented link admits a (non-canonical) compatible oriented decorated link. Likewise, every oriented decorated link admits a (non-canonical) compatible disoriented link.
\end{remark}

If two oriented decorated links $\olink_1$ and $\olink_2$ are compatible with the disoriented link $\dlink$, then $\HFLum(\olink_1)$ and $\HFLum(\olink_2)$ are canonically isomorphic.
Thus, we can define $\HFLum(\dlink)$ as $\HFLum(\olink)$, for any $\olink$ compatible with $\dlink$.
(More precisely, $\HFLum(\dlink)$ is the transitive system over all compatible oriented decorated links.)

%{\color{olive}
%\MM{Strike or keep this olive paragraph?}
Note that $\HFLum(\dlink)$ does not depend on the orientation chosen on $L$.
If $(L, \w, \z)$ is a compatible oriented decorated link, then the orientation reversal $L^r$ is also part of a compatible oriented decorated link, namely $(L^r, \z, \w)$. The swap of the $\w$ and $\z$ basepoints does not affect the homology, since the differential was defined to be symmetric in $\w$ and $\z$ (cf.~Equation \eqref{eq:CFK'differential}).
This justifies the name \emph{unoriented knot Floer homology} used in \cite{tHFK}.
%}

\begin{remark}
Fan~\cite{Fan} defines other categories of unoriented links, which he calls \emph{bipartite links} and \emph{bipartite disoriented links}.
These are essential to define a TQFT framework for disoriented links, but we do not recall them here.
%We instead list the properties we need and use them in an axiomatic way.
\end{remark}

%{\color{red}
%ADD PROPERTIES HERE.
%}

We now revise the cobordism maps defined by Fan~\cite{Fan}.

\begin{definition}
\label{def:dcob}
A \emph{disoriented link cobordism} from $\dlink_1=(L_1, \p_1, \q_1)$ to $\dlink_2=(L_2, \p_2, \q_2)$ is a pair $\dcob = (\Sigma, \motion)$ such that
\begin{enumerate}
\item $\Sigma \subset I \times S^3$ is a properly embedded, compact surface with $\Sigma \cap \{0\} \times S^3 = \{0\} \times (-L_1)$ and $\Sigma \cap \{1\} \times S^3 = \{1\} \times L_2$;
\item $\motion \subset \Sigma$ is a properly embedded, compact, oriented 1-manifold, which we we refer to as the \emph{motion} of the critical points;
\item \label{it:dlc-orientation} the components of $\Sigma \sm \motion$ are compact, oriented surfaces with orientation induced by $\motion$;
\item $\de\motion = \q_1-\p_1+\p_2-\q_2$.
\end{enumerate}
Note that, with the orientation given in point \eqref{it:dlc-orientation}, $\de (\Sigma \sm \motion) = \l_2 -\l_1 + 2\motion$. The surface $\Sigma$ does not need to be oriented.
\end{definition}

There is a natural notion of identity cobordism, in the same spirit as Definition \ref{def:idcob}.
We do not write such a definition explicitly.

By the work of Fan~\cite{Fan}, an disoriented link cobordism $\dcob$ from $\dlink_1$ to $\dlink_2$ induces an $\F[U]$-equivariant map
\[
\Hmap_{\dcob} \colon \HFLum(\dlink_1) \to \HFLum(\dlink_2).
\]

\begin{remark}
The map $\Hmap$ enjoys the following properties:
\begin{enumerate}
\item $\Hmap_{\dcob}$ is invariant under isotopy of $\Sigma$ in $I \times S^3$ while fixing the boundary, and under isotopy of $\motion$ in $\Sigma$ while fixing the boundary.
\item If $\id_{\dlink}$ is the identity cobordism from $\dlink$ to itself, then
\[
\Hmap_{\id_{\dlink}} = \id_{\HFLum(\dlink)}.
\]
\item If $\dcob_1$ and $\dcob_2$ are oriented disoriented link cobordisms from $\dlink_1$ to $\dlink_2$ and from $\dlink_2$ to $\dlink_3$ respectively, then one can stack $\dcob_2$ on top of $\dcob_1$, and obtain a new oriented disoriented link cobordism $\dcob_2 \circ \dcob_1$ from $\dlink_1$ to $\dlink_3$. In such a case,
\[
\Hmap_{\dcob_2 \circ \dcob_1} = \Hmap_{\dcob_2} \circ \Hmap_{\dcob_1}.
\]
\end{enumerate}
\end{remark}

\subsection{Relation between Zemke's TQFT and Fan's TQFT}

\begin{definition}
For $i=1,2$, suppose that $\olink_i=(L_i, \w_i, \z_i)$ and $\bdlink_i$ are compatible.
We say that a decorated link cobordism $\dlcob = (\Sigma, \dec)$ from $\olink_1$ to $\olink_2$ and a disoriented link cobordism $\dcob = (\Sigma, \motion)$ from $\bdlink_1$ to $\bdlink_2$ are \emph{compatible} if
\begin{itemize}
\item the underlying unoriented surface $\Sigma$ is the same (but note that in $\dlcob$ it also comes with an orientation);
\item after isotoping $\dec$ without crossing $\w_1 \sqcup \z_1 \sqcup \w_2 \sqcup \z_2$, $\dec = \motion$.
\end{itemize}
\end{definition}

\begin{remark}
For $i=1,2$, suppose that $\olink_i$ and $\dlink_i$ are compatible. Moreover, suppose that $\dlcob$ is a decorated link cobordism from $\olink_1$ to $\olink_2$ and $\dcob$ is a compatible disoriented link cobordism from $\dlink_1$ to $\dlink_2$. Then%, under the canonical isomorphism from Remark \ref{rem:canonicaliso},
\[
\Zmap_{\dlcob} = \Hmap_{\dcob}.
\]
\end{remark}

\section{A technical result for $\HFKum$}
\label{sec:techHFK}

\subsection{The flip cobordism in $\HFKum$}

\begin{definition}
The standard disoriented unknot is $\du=(U_1, p, q)$, where $U_1 = \set{x^2+y^2=1} \cap \set{z=0}$, $p=(1,0,0)$ and $q=(-1,0,0)$.
\end{definition}

\begin{definition}
The \emph{flip cobordism} $\flip=(\Sigma_{\flip}, \motion_{\flip})$ from the standard disoriented unknot $\du = (U_1, p, q)$ to itself is the disoriented cobordism traced by the isotopy obtained by rotating $U_1$ by $\pi$ along the $x$-axis. The points $p$ and $q$ stay fixed throughout the isotopy, so we can set $\motion_{\flip} = I \times \set{p, q}$.
\end{definition}

Note that the surface underlying a flip cobordism is orientable, although no orientation of the surface restricts to the same orientation on the two standard disoriented unknots on the boundary.

\begin{figure}
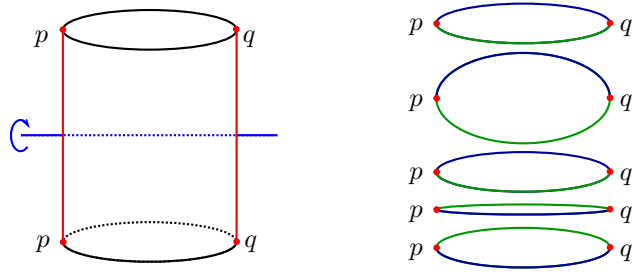
\caption{The figure on the left is our notation for the flip cobordism. On the right we sketched a few sections of the cobordism. We used different colours for the two components of $U_1 \setminus\set{p,q}$ to help the visualisation.}
\end{figure}

\begin{lemma}
\label{lem:flip}
%\MM{Uhm, maybe there are too many F's in $\Hmap_{\flip}$. Suggestions are welcome.}
The map $\Hmap_{\flip}$ induced by the flip cobordism is the identity map on $\HFKum(\du) \cong \F[U]$.
\end{lemma}
\begin{proof}
%\MM{Is the figure any better? Should I just get rid of the right hand side?}
%\MM{\sout{TQFT proof that would work over $\Z$ coefficients: take pre-composition with birth or post-composition with death (cobordism from or to $\varnothing$ link in the $\varnothing$ manifold).} Haofei does not explicitly say it, but in my opinion his links must be non-empty.}
The fourth iteration $\flip^4$ is the disoriented cobordism traced by a $4\pi$ rotation about the $x$-axis. Since $\pi_1(SO(3)) = \Z/2\Z$, the rotation by $4\pi$ is isotopic to the identity. Thus, $\flip^4$ is isotopic to the identity cobordism, and
\begin{equation}
\label{eq:fourthpower}
\left(\Hmap_{\flip}\right)^4 = \Hmap_{\flip^4} = \id_{\F[U]}.
\end{equation}
The map
\[
\Hmap_{\flip} \colon \F[U] \to \F[U]
\]
is $U$-equivariant, so it is completely determined by the image of $1$. If we set $p(U) := \Hmap_{\flip}(1) \in \F[U]$, Equation \eqref{eq:fourthpower} implies that $\left(p(U)\right)^4 = 1$. Since every invertible element of $\F[U]$ must be in $\F$, we deduce $p(U) = 1$.
\end{proof}

\subsection{A stabilisation lemma}
In this subsection only, we will need to work in a more general setting than the one outlined in Section \ref{sec:bgHFK}.

First, we will consider decorated links $\olink$ in a 3-manifold $Y$, and decorated link cobordisms $(\Sigma, \mc A)$ in a 4-manifold $W$.
In Section \ref{sec:bgHFK}, we have stated the definitions of decorated link and decorated link cobordism only when $Y = S^3$ and $W = I \times S^3$. The general definitions are only needed in this subsection, and they can be found in \cite[Definitions 2.1 and 2.4]{ZFunctoriality}.
Again in this subsection only, we will consider the full chain complex $\CFLcm(Y, \oknot)$ associated to a decorated knot, which is a complex over $\F[U, V]$, up to chain homotopy equivalence.

We will also use the homological action on link Floer homology. (See \cite[Section 4.2.5]{HF} for the original definition of Heegaard Floer homology, \cite[Theorem 3.1]{OSzTriangles} for the cobordism action, and \cite[Section 12.2]{ZGradings} for its extension to link Floer homology.)
Given a decorated link $\olink$ in a 3-manifold $Y$ there is a homological action
\[
A \colon \Lambda^*(H_1(Y;\Z)/\mathrm{Tors}) \otimes \CFLcm(Y, \olink) \to \CFLcm(Y, \olink),
\]
and given a decorated link cobordism $(\Sigma, \mc A)$ in a 4-manifold $W$ there is a version of the cobordism map incorporating the homological action
\[
F^H_{W, \Sigma, \mc A} \colon \Lambda^*(H_1(W;\Z)/\mathrm{Tors}) \otimes \CFLcm(Y_1, \olink_1) \to \CFLcm(Y_2, \olink_2).
\]
(We use the notations $F^H$ to distinguish it from the cobordism map $\Zmap$ which does not incorporate the homological action.)
If the 4-manifold $W$ is obtained by adding 1-handles to $B^4$, then the map $F^H$ can be recovered from $\Zmap$ by post-composing with the homological action on $Y_1$.
The following lemma, which is needed to establish Proposition \ref{prop:HFK}, was proved by Ian Zemke. A related argument appeared in \cite[Section 5]{JZClasp} (see in particular \cite[Lemma 5.3]{JZClasp}).

\begin{lemma}
\label{lem:Ian}
Let $\Sigma = I \times K \subset I \times S^3$ be the identity cobordism from the knot $K$ to itself, and let $\Sigma'$ denote a surface obtained by adding a compressible 1-handle to $\Sigma$.
If $\gamma \subset \Sigma$ denotes an embedded arc joining the feet of the 1-handle, define decorations $\mathcal A'$ on $\Sigma'$ as two parallel embedded arcs from $\set0 \times K$ to $\set1 \times K$ such that:
\begin{itemize}
\item $\mathcal A'$ does not intersect $\gamma$;
\item each arc of $\mathcal A'$ crosses the co-core of the 1-handle exactly once;
\item the arcs of $\mathcal A'$ join the points $(0, p)$ and $(0,q)$ to the points $(1,p)$ and $(1,q)$ in $I \times K$ respectively;
\item the decorations $\mathcal A$, obtained by restricting $\mathcal A'$ to $\Sigma$ and by reconnecting each pair of arcs with an arc parallel to $\gamma$, is isotopic rel boundary to a product decoration $I \times \set{p, q}$.
\end{itemize}
Then, if $\oknot = (K, w, z)$ for some points $w, z$ alternated to $p, q$, the cobordism map
\[
\Zmap_{\Sigma', \mathcal A'} \colon \HFKum(\oknot) \to \HFKum(\oknot)
\]
coincides with the map $U \cdot \id_{\HFKum(\oknot)}$.
\end{lemma}

\begin{proof}
If $a_1$ and $a_2$ denote the two components of $\mc A'$, let $\delta$ be an arc on $\Sigma'$ which starts from $a_1$ near a foot of the 1-handle, then traverses $a_2$, follows $\gamma$ to the other foot of the 1-handle, and ends on $a_2$. See Figure \ref{fig:bypass} on the left for an illustration.

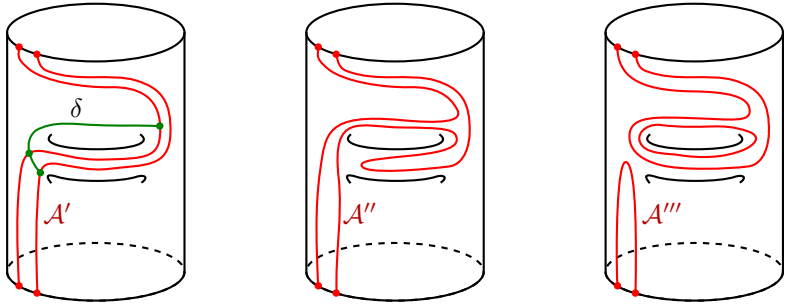
\begin{figure}
%\resizebox{\textwidth}{!}{
%\includegraphics{bypass.jpg}
%}
%% Creator: Inkscape 1.0.1 (c497b03c, 2020-09-10), www.inkscape.org
%% Accompanies image file 'bypass.pdf' (pdf, eps, ps)
\begingroup%
  \makeatletter%
  \providecommand\color[2][]{%
    \errmessage{(Inkscape) Color is used for the text in Inkscape, but the package 'color.sty' is not loaded}%
    \renewcommand\color[2][]{}%
  }%
  \providecommand\transparent[1]{%
    \errmessage{(Inkscape) Transparency is used (non-zero) for the text in Inkscape, but the package 'transparent.sty' is not loaded}%
    \renewcommand\transparent[1]{}%
  }%
  \providecommand\rotatebox[2]{#2}%
  \newcommand*\fsize{\dimexpr\f@size pt\relax}%
  \newcommand*\lineheight[1]{\fontsize{\fsize}{#1\fsize}\selectfont}%
  \ifx\svgwidth\undefined%
    \setlength{\unitlength}{293.16547239bp}%
    \ifx\svgscale\undefined%
      \relax%
    \else%
      \setlength{\unitlength}{\unitlength * \real{\svgscale}}%
    \fi%
  \else%
    \setlength{\unitlength}{\svgwidth}%
  \fi%
  \global\let\svgwidth\undefined%
  \global\let\svgscale\undefined%
  \makeatother%
  \begin{picture}(1,0.38462144)%
    \lineheight{1}%
    \setlength\tabcolsep{0pt}%
    \put(0,0){\includegraphics[width=\unitlength,page=1]{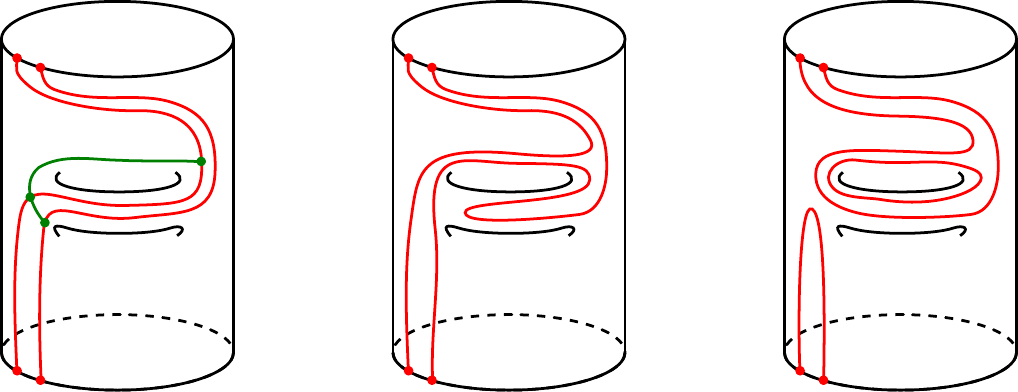}}%
    \put(0.04812867,0.10115897){\color[rgb]{0,0,0}\makebox(0,0)[lt]{\lineheight{1.25}\smash{\begin{tabular}[t]{l}\small \color{red}$\mathcal{A}'$\end{tabular}}}}%
    \put(0.08264525,0.23778675){\color[rgb]{0,0,0}\makebox(0,0)[lt]{\lineheight{1.25}\smash{\begin{tabular}[t]{l}\small $\delta$\end{tabular}}}}%
    \put(0.43274101,0.10116277){\color[rgb]{0,0,0}\makebox(0,0)[lt]{\lineheight{1.25}\smash{\begin{tabular}[t]{l}\small \color{red}$\mathcal{A}''$\end{tabular}}}}%
    \put(0.81741374,0.10116276){\color[rgb]{0,0,0}\makebox(0,0)[lt]{\lineheight{1.25}\smash{\begin{tabular}[t]{l}\small \color{red}$\mathcal{A}'''$\end{tabular}}}}%
  \end{picture}%
\endgroup%

\caption{The three figures show the three decorations $\mc A'$, $\mc A''$, and $\mc A'''$ on the surface $\Sigma'$ appearing in the bypass relation arising from the arc $\delta$.}
\label{fig:bypass}
\end{figure}

We apply Zemke's bypass relation on a disc $\Delta \subset \Sigma'$ obtained as a regular neighbourhood of the arc $\delta$ in $\Sigma'$. If $\mc A''$ and $\mc A'''$ denote the other decorations appearing in the bypass relation as in Figure \ref{fig:bypass}, we have that
\[
\Zmap_{\Sigma', \mathcal A'} = \Zmap_{\Sigma', \mc A''} + \Zmap_{\Sigma', \mc A'''}.
\]

The decorations $\mc A''$ can be isotoped away from the co-core of the 1-handle. After compressing the 1-handle the surface becomes isotopic to $\Sigma$ and the decorations become isotopic to the product decorations $\mc A$. Thus, by Remark \ref{rem:propertiesZ}.\eqref{it:compression},
\[
\Zmap_{\Sigma', \mc A''} = U \cdot \Zmap_{\Sigma, \mc A} = U \cdot \id_{\HFKum(\oknot)}.
\]

Thus, we only need to show that $\Zmap_{\Sigma', \mc A'''}=0$. From this point until the end of the proof we will work on the chain level $\CFLcm(\oknot)$, considered as an $\F[U, V]$-complex, up to chain homotopy equivalence. (The variable $U$ is associated to the basepoint $w$.) 

We split the cobordism $(I \times S^3, \Sigma', \mc A''')$ as the composition of two cobordisms.
The first one, which we call $\mc W_1 = (W_1, \Sigma_1, \mc A_1)$, is obtained by taking as $W_1$ the (disjoint) union of a regular neighbourhood of $\set0 \times S^3$ and a neighbourhood of $\gamma \cup c$ (where $c$ denotes the core of the 1-handle) containing the 1-handle entirely. Note that the latter component of $W_1$ is diffeomorphic to $S^1 \times D^3$. The decorated surface $(\Sigma_1, \mc A_1)$ is obtained by intersecting $W_1$ with $(\Sigma', \mc A''')$.
The second cobordism $\mc W_2$ is obtained by taking the closure of the complement of $W_1$ in $I \times S^3$. Thus, we have
\[
\Zmap_{\Sigma, \mc A'''} =
F^H_{(I \times S^3, \Sigma', \mc A''')} = F^H_{\mc W_2} \circ F^H_{\mc W_1}.
\]
(In the first equality we used the fact that $H_1(I \times S^3) = 0$.)

We focus on the map $F^H_{\mc W_1}$. Since $W_1$ has two connected components (one of which is an identity cobordism over $\oknot$), the map splits as a tensor product
\begin{equation}
\label{eq:v2,1}
F^H_{\mc W_1} = \id_{\CFLcm(\oknot)} \otimes F^H_{\widetilde{\mc W_1}},
\end{equation}
where $\widetilde{\mc W_1} = (S^1 \times D^3, \widetilde{\Sigma_1}, \widetilde{\mc A_1})$ is a cobordism from the empty link in the empty 3-manifold to a doubly pointed unknot $\mathbb U$ in $S^1 \times S^2$, illustrated in Figure \ref{fig:Ian} on the left.

\begin{figure}
%\resizebox{\textwidth}{!}{
\includegraphics{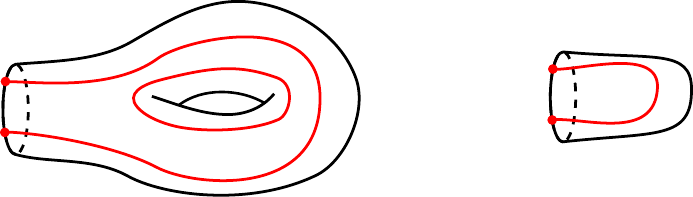}
%}
\caption{The figure on the left shows the decorated surface $(\widetilde{\Sigma_1}, \widetilde{\mc A_1})$. The punctured torus $\widetilde{\Sigma_1}$ sits in $S^1 \times D^3$ in such a way that its longitude generates $H_1(S^1 \times D^3)$, while its meridian is null-homotopic. The figure on the right shows the decorated surface $(\overline{\Sigma_1}, \overline{\mc A_1})$.}
\label{fig:Ian}
\end{figure}

The knot Floer complex $\CFLcm(S^1 \times S^2, \mathbb U)$ is generated over $\F[U, V]$ by two homogeneous elements $\x_+$ and $\x_-$. Their $\gr_\w$ and $\gr_\z$ gradings (as defined in \cite{ZGradings}) are given by the formulas
\[
(\gr_\w, \gr_\z)(\x_\pm) = \left(\pm\frac12, \pm\frac12\right).
\]
By grading reasons \cite{ZGradings}, we have that
\begin{equation}
\label{eq:v2,2}
F^H_{\widetilde{\mc W_1}}(1) = k \cdot \x_-
\end{equation}
for some $k \in \Z/2\Z$. An explicit computation of the action of $H_1(S^1 \times S^2) = \Z\langle{\zeta}\rangle$ shows that $A(\zeta \otimes \x_+) = \x_-$. From this fact, a direct computation shows that
\begin{equation}
\label{eq:v2,3}
F^H_{\overline{\mc W_1}}(\zeta \otimes 1) = \x_-,
\end{equation}
where $\overline{\mc W_1} = (S^1 \times S^2, \overline{\Sigma_1}, \overline{\mc A_1})$ is the cobordism shown in Figure \ref{fig:Ian} on the right.

Recall that the cobordism $\mc W_1$ is the disjoint union of an identity cobordism over $\oknot$ and the cobordism $\widetilde{\mc W_1}$. If $\widehat{\mc W_1}$ denotes the cobordism obtained by replacing the $\widetilde{\mc W_1}$ component with $\overline{\mc W_1}$, then by combining Equations \eqref{eq:v2,1}, \eqref{eq:v2,2}, and \eqref{eq:v2,3} we have
\begin{equation}
\label{eq:v2,4}
F^H_{\mc W_1}(\x) = 
\x \otimes F^H_{\widetilde{\mc W_1}}(1) = 
k \cdot \x \otimes F^H_{\overline{\mc W_1}}(\zeta \otimes 1) =
k \cdot F^H_{\widehat{\mc W_1}}(\zeta \otimes \x).
\end{equation}

Finally, let $\widehat{\mc W}$ denote the composition of $\widehat{\mc W_1}$ and $\mc W_2$. Note that the 4-manifold underlying $\widehat{\mc W}$ is still $I \times S^3$ (same as $\mc W$), since the replacement of ${\mc W_1}$ with $\widehat{\mc W_1}$ did not affect the underlying 4-manifold.
Then, by Equation \eqref{eq:v2,4},
\begin{align*}
\Zmap_{\Sigma, \mc A'''} (\x) &=
F^H_{\mc W_2} \circ F^H_{\mc W_1} (\x)\\
&= k \cdot F^H_{\mc W_2} \circ F^H_{\widehat{\mc W_1}}(\zeta \otimes \x)\\
&= k \cdot F^H_{\widehat{\mc W}}(\iota_*(\zeta) \otimes \x) = 0.
\end{align*}
The last term vanishes because the map $\iota_* \colon H_1(S^1 \times D^3) \to H_1(I \times S^3) = 0$ induced by the inclusion of $\widehat{\mc W_1}$ into $\widehat{\mc W}$ must map $\zeta$ to $0$.
\end{proof}

\subsection{The main theorem in $\HFKum$}

\begin{proposition}
\label{prop:HFK}
Suppose that $\Sigma$ is a connected unorientable knot cobordism from $K_1$ to $K_2$ in $I \times S^3$ with $m$ local minima, $b$ saddles, and $M$ local maxima, and let $\m{\Sigma}$ denote the mirrored upside down cobordism from $K_2$ to $K_1$.

Then there are choices of motions of the critical points such that the disoriented knot cobordisms $\dcob=(\Sigma, \motion_1)$ and $\m\dcob=(\m\Sigma, \motion_2)$ can be composed to $\m\dcob \circ \dcob$, and
\begin{equation}
\label{eq:HFK}
U^M \cdot \Hmap_{\m\dcob} \circ \Hmap_{\dcob} = U^{b-m} \cdot \id_{\HFKum(K_1)}.
\end{equation}
\end{proposition}

\begin{proof}
Using Lemma \ref{lem:break}, we can break the cobordism $\Sigma$ into the composition of cobordisms labelled \eqref{it:B}-\eqref{it:D}. Let $K_1'$ and $K_2'$ be the knots after steps \eqref{it:BS} and \eqref{it:US} respectively, as in the statement of Lemma \ref{lem:break}, and let $L'$ be the link after step \eqref{it:OS}.
Note that $L'$ differs from $K_2'$ by a single band surgery.

\begin{figure}
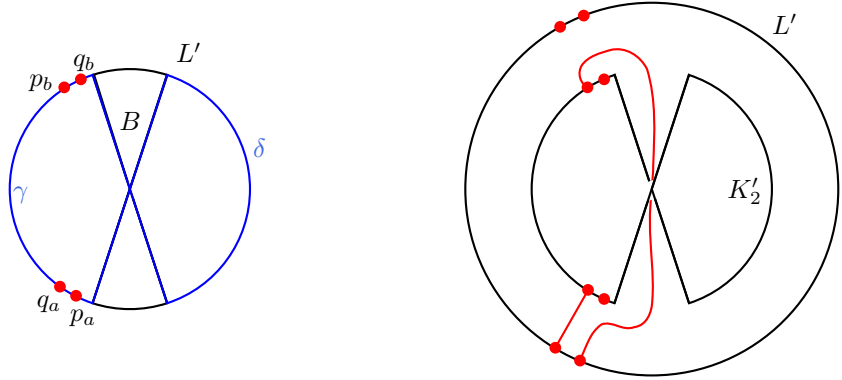
\caption{The circle on the left hand side represents $L'$ after step \eqref{it:OS}, in case it is a knot. The figure on the right hand side shows step \eqref{it:US} of the cobordism $\Sigma$, from $L'$ to $K_2'$, together with the motion chosen to define $\dcob$.}
\label{fig:step4}
\end{figure}

By removing the two attaching arcs of the unorientable band $B$ from $L'$, we are left with two arcs $\gamma$ and $\delta$.
%\MM{Fixed order of the basepoints.}
If $L'$ is a knot, let $p_a, q_a, p_b, q_b$ be points on $\gamma$, appearing in this order, such that $p_a$ and $q_a$ are close to one end of $\gamma$ and $p_b$ and $q_b$ are close to the other end of $\gamma$, so that all the intersections of $L' \cap \gamma$ with the oriented bands are between $p_b$ and $q_a$. See the left hand side of Figure \ref{fig:step4}. If instead $L'$ is a 2-component link, let $p_a$ and $q_a$ be on $\gamma$ and $p_b$ and $q_b$ be on $\delta$, such that they are near opposite corners of the band, and $p_a$ (resp.\ $q_b$) is closer to the band than $q_a$ (resp.\ $p_b$). See Figure \ref{fig:step4link}.

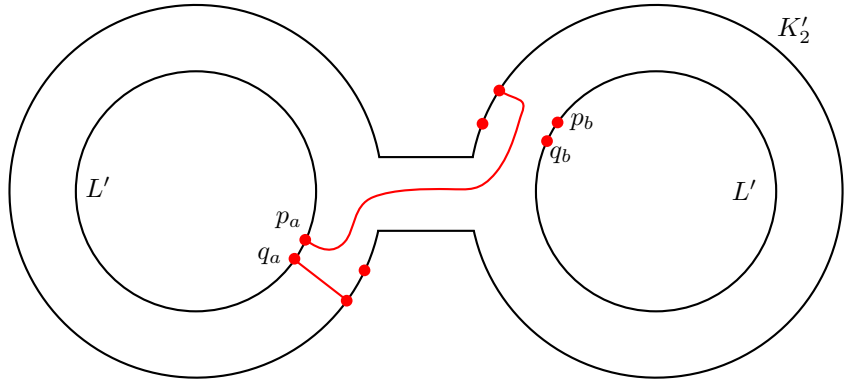
\begin{figure}
%% Creator: Inkscape 1.0.1 (c497b03c, 2020-09-10), www.inkscape.org
%% Accompanies image file 'step4link.pdf' (pdf, eps, ps)
\begingroup%
  \makeatletter%
  \providecommand\color[2][]{%
    \errmessage{(Inkscape) Color is used for the text in Inkscape, but the package 'color.sty' is not loaded}%
    \renewcommand\color[2][]{}%
  }%
  \providecommand\transparent[1]{%
    \errmessage{(Inkscape) Transparency is used (non-zero) for the text in Inkscape, but the package 'transparent.sty' is not loaded}%
    \renewcommand\transparent[1]{}%
  }%
  \providecommand\rotatebox[2]{#2}%
  \newcommand*\fsize{\dimexpr\f@size pt\relax}%
  \newcommand*\lineheight[1]{\fontsize{\fsize}{#1\fsize}\selectfont}%
  \ifx\svgwidth\undefined%
    \setlength{\unitlength}{314.71710071bp}%
    \ifx\svgscale\undefined%
      \relax%
    \else%
      \setlength{\unitlength}{\unitlength * \real{\svgscale}}%
    \fi%
  \else%
    \setlength{\unitlength}{\svgwidth}%
  \fi%
  \global\let\svgwidth\undefined%
  \global\let\svgscale\undefined%
  \makeatother%
  \begin{picture}(1,0.44914384)%
    \lineheight{1}%
    \setlength\tabcolsep{0pt}%
    \put(0,0){\includegraphics[width=\unitlength,page=1]{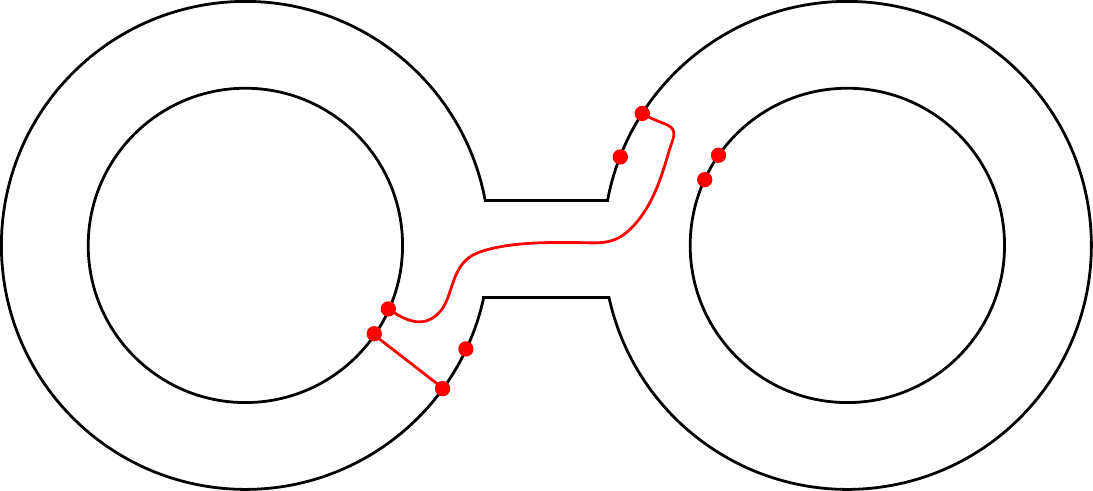}}%
    \put(0.92042096,0.41037653){\color[rgb]{0,0,0}\makebox(0,0)[lt]{\lineheight{0}\smash{\begin{tabular}[t]{l}\small $K_2'$\end{tabular}}}}%
    \put(0.86607843,0.21428213){\color[rgb]{0,0,0}\makebox(0,0)[lt]{\lineheight{0}\smash{\begin{tabular}[t]{l}\small $L'$\end{tabular}}}}%
    \put(0.2983268,0.14344225){\color[rgb]{0,0,0}\makebox(0,0)[lt]{\lineheight{0}\smash{\begin{tabular}[t]{l}\small $q_a$\end{tabular}}}}%
    \put(0.3202817,0.18666787){\color[rgb]{0,0,0}\makebox(0,0)[lt]{\lineheight{0}\smash{\begin{tabular}[t]{l}\small $p_a$\end{tabular}}}}%
    \put(0.67307025,0.30377996){\color[rgb]{0,0,0}\makebox(0,0)[lt]{\lineheight{0}\smash{\begin{tabular}[t]{l}\small $p_b$\end{tabular}}}}%
    \put(0.64801372,0.26419782){\color[rgb]{0,0,0}\makebox(0,0)[lt]{\lineheight{0}\smash{\begin{tabular}[t]{l}\small $q_b$\end{tabular}}}}%
    \put(0.0923841,0.21791345){\color[rgb]{0,0,0}\makebox(0,0)[lt]{\lineheight{0}\smash{\begin{tabular}[t]{l}\small $L'$\end{tabular}}}}%
  \end{picture}%
\endgroup%

\caption{The figure shows step \eqref{it:US} of $\dcob$ in case $L'$ is a 2-component link, represented above by the two inner circles.}
\label{fig:step4link}
\end{figure}

Let $\dcob$ be the disoriented cobordism from $(K_1, p_a, q_a)$ to $(K_2, p_b, q_a)$ obtained by endowing $\Sigma$ with the following motion of basepoints:
\begin{itemize}
\item on steps \eqref{it:B}-\eqref{it:OS}, the motion consists of straight arcs $I \times \set{p_a, q_a}$;
\item on step \eqref{it:US}, the motion consists of a straight arc $I \times \set{q_a}$ and of an arc that starts from $p_a$, goes through the unorientable saddle, and ends at $p_b$ (see Figure \ref{fig:step4} for an illustration);
\item on steps \eqref{it:DS}-\eqref{it:D}, the motion consists of straight arcs $I \times \set{p_b, q_a}$.
\end{itemize}
A crucial condition in Definition \ref{def:dcob} is that each component of $\Sigma\sm\motion_1$ must be orientable. In fact, we show that $\Sigma \setminus \motion_1$ consists of a single and orientable component. If $L'$ is a knot, one can check from Figure \ref{fig:step4} that $\Sigma\setminus\motion_1$ restricted to step \eqref{it:US} has a single component; in steps \eqref{it:B}-\eqref{it:OS} the surface $\Sigma$ is orientable and the motion is given by two parallel arcs, so there are two components of $\Sigma\setminus\motion_1$, which are then glued to the unique component in step \eqref{it:US}; steps \eqref{it:DS}-\eqref{it:D} define a concordance of disoriented knots, which does not change the abstract topology of the disoriented cobordism.
The compatibility of the orientation of $\Sigma\sm\motion_1$ with the orientation of $\motion_1$ is dealt with in an analogous way.

%\MM{I have expanded on the 2-component case. Is the paragraph too long now?}
If instead $L'$ is a 2-component link, then one should consider Figure \ref{fig:step4link} instead. Let $\zeta$ be the closed component of $L'$ containing $p_b$ and $q_b$ (which appears on the right in Figure \ref{fig:step4link}), and let $\varepsilon$ be the component containing $p_a$ and $q_a$, minus the short arc connecting $p_a$ and $q_a$ (which appears on the left in Figure \ref{fig:step4link}). From Figure \ref{fig:step4link} it is immediate to check that $\Sigma\setminus\motion_1$ has 2 components in step \eqref{it:US}, which deformation retract on $\zeta$ and on $\varepsilon$.% Call them $Z$ and $E$ respectively.
Since $K_1$ is a knot, $\Sigma\setminus\motion_1$ in steps \eqref{it:B}-\eqref{it:OS} also has two components: a ``small'' rectangular one ($S$), spanned by the short arc connecting $p_a$ and $q_a$, and a large one ($L$) the complement of it, which contains all the genus. When you glue steps \eqref{it:B}-\eqref{it:OS} to step \eqref{it:US}, the rectangular component $S$ is glued to the component containing $\zeta$, without affecting the topology, whereas the component $L$ glues to both component of step \eqref{it:US}. Thus, we see that there is only one component of $\Sigma\setminus\motion_1$. Its orientability and the compatibility with the orientation of $\motion_1$ is left to the reader (it basically follows from the fact that cutting along $\motion_1$ effectively cuts the unorientable saddle, leaving an orientable cobordism). As before, we do not worry about steps \eqref{it:DS}-\eqref{it:D}, since they define a concordance, which does not change the abstract topology.

We next introduce a disoriented cobordism $\m\dcob$ from $(K_2, p_b, q_a)$ to $(K_1, p_a, q_a)$ with underlying surface $\m\Sigma$.
To define it, we play the steps of the cobordism $\dcob$ in reverse order, but we use a different motion of basepoints:
\begin{itemize}
\item on the reversed steps \eqref{it:D}-\eqref{it:DS}, the motion consists of straight arcs $I \times \set{p_b, q_a}$;
\item on the reversed step \eqref{it:US}, the motion consists of a straight arc $I \times \set{p_b}$ and of an arc that starts from $q_a$, goes through the (dual) unorientable saddle, and ends at $q_b$;
\item on the reversed steps \eqref{it:OS}-\eqref{it:B}, the motion consists of straight arcs $I \times \set{p_b, q_b}$;
\item finally, in a collar of the $K_1$ boundary component, the motion of the basepoints brings $p_b$ and $q_b$ back to $p_a$ and $q_a$.
\end{itemize}
Note that $\m\dcob$ is not $\dcob$ turned upside down as disoriented cobordisms (even if the disoriented knots at the boundary are not the same). 

We also define a disoriented cobordism $\tdcobu$ from $(K_1, p_a, q_a)$ to $(K_1, p_a, q_a)$, obtained in three steps:
\begin{itemize}
\item the first step is the same disoriented cobordism as in Figure \ref{fig:step4}, except that the knot $L'$ is replaced with $K_1$; more explicitly, the surface $\Sigma$ in the first step consists of the cylinder $I \times K_1$, with the unorientable band $B$ attached on the upper end (recall that by Lemma \ref{lem:break} all bands have disjoint attaching arcs, so we can attach $B$ to $K_1$), and the motion of the basepoints consists of $I \times \set{q_a}$ and of an arc that starts from $p_a$, goes through the band, and ends at $p_b$;
\item the surface in the second step is simply the surface from the first step turned upside down, and the motion consists of a straight arc $I \times \set{p_b}$ and of an arc that starts from $q_a$, goes through the (dual) band, and ends at $q_b$;
\item finally, in a collar of the end boundary component, the motion of the basepoints brings $p_b$ and $q_b$ back to $p_a$ and $q_a$.
\end{itemize}
Note that the surface $\Sigma$ of the disoriented cobordism $\tdcobu$ is a genus-1 surface, consisting of a cylinder $I \times K_1$ with a flipped tube attached to it. The flipped tube is made up of the two unorientable bands. See the left hand side of Figure \ref{fig:tdcob}.

\begin{figure}
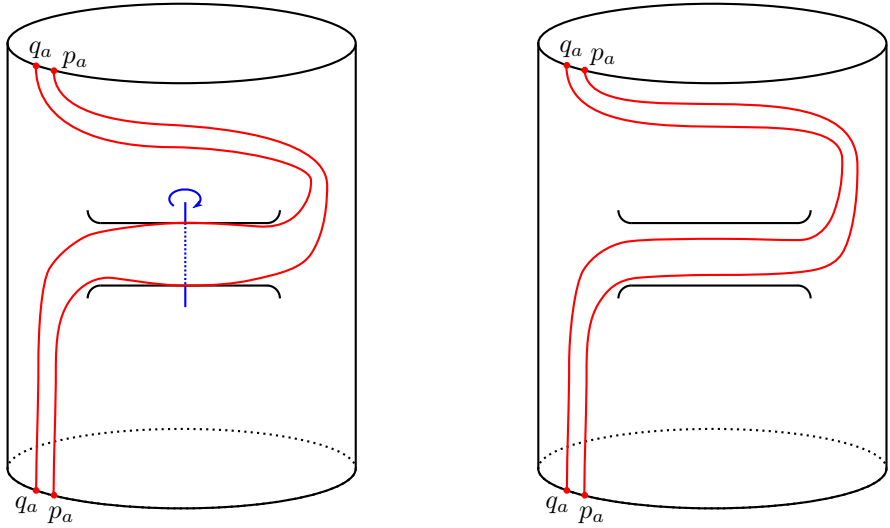
\caption{The figure on the left represents the disoriented cobordism $\tdcobu$. The figure on the right represents the disoriented cobordism $\tdcobo$.}
\label{fig:tdcob}
\end{figure}

Lastly, we define a variant of $\tdcobu$:
%\begin{itemize}
%\item the disoriented cobordism $\mtdcobu$ from $(K_1, p_b, q_b)$ to $(K_1, p_a, q_a)$ is obtained by turning $\tdcobu$ upside down;
%\item 
the disoriented cobordism $\tdcobo$ from $(K_1, p_a, q_a)$ to $(K_1, p_a, q_a)$ is obtained by replacing the flipped tube in $\tdcobu$ with an orientable tube, so that the underlying surface $\Sigma$ is orientable (in other words, the unorientable bands are replaced with orientable bands); the motion of the basepoints divides $\Sigma$ into a disc and a punctured torus; see the right hand side of Figure \ref{fig:tdcob}.
%\item the disoriented cobordism $\mtdcobo$ from $(K_1, p_b, q_b)$ to $(K_1, p_a, q_a)$ is obtained by turning $\tdcobo$ upside down.
%\end{itemize}
Note that Lemma \ref{lem:flip} implies that $\Hmap_{\tdcobu}=\Hmap_{\tdcobo}$,
%and $\Hmap_{\mtdcobu}=\Hmap_{\mtdcobo}$,
since it is possible to isolate a flip cobordism.

In order to prove Proposition \ref{prop:HFK} we argue in a similar way as in \cite[Proposition 4.1]{JMZ}: we define a cobordism $\gdcobu$, and we compute the map $\Hmap_\gdcobu$ in two different ways, which will be the two sides of equation \eqref{eq:HFK}.

The disoriented cobordism $\gdcobu$, from $(K_1, p_a, q_a)$ to itself, is obtained by playing all the steps of $\dcob$ except \eqref{it:D} followed by all the reversed steps of $\m\dcob$ except \eqref{it:D}. (In $\m\dcob$ we also play the basepoint moving step in a collar of $K_1$.
%\begin{itemize}
%\item steps \eqref{it:B}-\eqref{it:DS} of $\dcob$;
%\item the reversed steps \eqref{it:DS}-\eqref{it:B} of $\m\dcob$.
%%\item the disoriented cobordism $\mtdcobo$.
%\end{itemize}

Since the attaching arc of the unoriented band can be isotoped in $K_1$, we can move step \eqref{it:US} of $\dcob$ and the corresponding reversed step of $\m\dcob$ up past all the other steps of $\dcob$ and $\m\dcob$.
%, just before $\mtdcobo$.
These two steps together make up the cobordism $\tdcobu$, which we can replace with the orientable cobordism $\tdcobo$. The replacement yields a new disoriented cobordism $\gdcobo$ from $\gdcobu$, with $\Hmap_{\gdcobo} = \Hmap_{\gdcobu}$. The advantage of $\gdcobo$ over $\gdcobu$ is that the underlying surface of the former is orientable, so $\Hmap_{\gdcobo} = \Zmap_{\gdlcob}$ for a compatible decorated link cobordism $\gdlcob$, and we can use the properties of Zemke's TQFT on $\Zmap_{\gdlcob}$, in particular the one about compressing discs.

Note that in the definition of $\gdcobu$ (or $\gdcobo$) we do not play the $M$ deaths of $\dcob$ (step \eqref{it:D}) and the $M$ births of $\m\dcob$, obtained by mirroring the deaths of $\dcob$. Thus, $\m{\dcob} \circ \dcob$ is obtained from $\gdcobu$ by compressing $M$ discs with boundary in the complement of the motion of the basepoints. By transiting through their orientable replacements, and by Remark \ref{rem:propertiesZ}.\eqref{it:compression}, we get
\begin{equation}
\label{eq:proofHFK1}
\Hmap_{\gdcobu} = U^M \cdot \circ \Hmap_{\m\dcob} \circ \Hmap_{\dcob}.
\end{equation}

On the other hand, we saw earlier that the cobordism $\gdcobo$ can be re-arranged so that we have $\tdcobo$ at the top. The first part consists of a cylindrical cobordism from $(K_1, p_a, q_a)$ to itself with $b-1-m$ tubes added, as in \cite{JMZ}. (The $-1$ here comes from the fact that we have moved the unorientable band to the top of the cobordism.) Thus, we can compress the cobordism $\gdcobo$ $b-1-m$ times to get $\tdcobo$, so
\[
\Hmap_{\gdcobo} = U^{b-1-m} \cdot \Hmap_{\tdcobo}.
\]
But the cobordism $\tdcobo$ is of the form studied in Lemma \ref{lem:Ian}, so the map that it induces is the multiplication by $U$. Thus:
\begin{equation}
\label{eq:proofHFK2}
\Hmap_{\gdcobu} = \Hmap_{\gdcobo} = U^{b-m} \cdot \id_{\HFLum(K_1)}.
\end{equation}
By combining Equations \eqref{eq:proofHFK1} and \eqref{eq:proofHFK2} we finish the proof.
\end{proof}

%\begin{figure}
%\input{doublecompression.pdf_tex}
%\caption{The figure shows that the cobordism $\mtdcobo \circ \tdcobo$ is compressible twice. The green dashed curves indicate the compressing circles.}
%\label{fig:doublecompression}
%\end{figure}

\begin{remark}
The careful reader will note that the motion of the basepoints play an important role in the proof of Proposition \ref{prop:HFK}. This is by contrast with \cite[Proposition 4.1]{JMZ}, where the decorations of the cobordism were the simplest possible, i.e.~two parallel arcs from the bottom to the top.
In the unoriented setting it is impossible to choose two parallel arcs as the motion of basepoints, otherwise the cobordism would not fall in the correct category.
\end{remark}

\section{Applications}
\label{sec:corollaries}

In this section we prove Theorem \ref{thm:main}, which we restate below, together with its corollaries.

{\renewcommand{\thetheorem}{\ref{thm:main}}
\begin{theorem}
%\label{thm:main}
Let $K_1$ and $K_2$ be knots in $S^3$. Suppose that there is an unorientable knot cobordism $\Sigma$ in $I \times S^3$ from $K_1$ to $K_2$ with $M$ local maxima. Then
{\renewcommand{\theequation}{\ref{eq:I_main}}
\begin{equation}
\OrdI(K_1) \leq \max\set{\OrdI(K_2), M} + \gamma(\Sigma)
%\OrdI(K_1) \leq \max\set{\OrdI(K_2), M} - \chi(\Sigma)
\end{equation}
\addtocounter{equation}{-1}
}
and
{\renewcommand{\theequation}{\ref{eq:U_main}}
\begin{equation}
\OrdU(K_1) \leq \max\set{\OrdU(K_2), M} + \gamma(\Sigma).
%\OrdU(K_1) \leq \max\set{\OrdU(K_2), M} - \chi(\Sigma) +1.
\end{equation}
\addtocounter{equation}{-1}
}
\end{theorem}
\addtocounter{theorem}{-1}
}

\begin{proof}
The proof closely follows that of \cite[Theorem 1.1]{JMZ}.

Add decorations on $\Sigma$ and $\overline\Sigma$ (in the instanton or unoriented knot Floer sense) to obtain cobordisms with decorations $\mathcal{S}$ and $\overline{\mathcal{S}}$ that satisfy the relation in Proposition \ref{propn:main_technical} or \ref{prop:HFK}
\begin{equation}
\label{eq:JMZ-main}
v^M \cdot F_{\overline {\mathcal S}} \circ F_{\mathcal S} = v^{b-m} \cdot \id_{H(K_1)}.
\end{equation}
Here $m$ is the number of local minima and $b$ is the number of saddles on $\Sigma$, $H$ denotes either $I^\sharp$ or $\HFKum$, $F$ denotes the corresponding map induced by an unoriented cobordisms with decorations, and $v$ denotes the relevant variable $P$ or $U$.

Suppose that $x \in H(K_1)$ is a torsion element. Since $F_{\mathcal S}(x)$ must be torsion in $H(K_2)$, $v^\ell \cdot F_{\overline {\mathcal S}} \circ F_{\mathcal S} (x) = F_{\overline {\mathcal S}} (v^\ell \cdot F_{\mathcal S} (x)) = 0$ whenever $\ell \geq \Ord(K_2)$.
Thus, in view of Equation \eqref{eq:JMZ-main}, $v^{\ell+b-m-M} \cdot x = 0$ whenever $\ell \geq \max\{\Ord(K_2), M\}$. Since this holds for every torsion element $x \in H(K_1)$, we obtain
\[
\Ord(K_1) \leq \max\{\Ord(K_2), M\} + b - m - M,
\]
and we conclude by noticing that $\gamma(\Sigma) = -\chi(\Sigma) = b-m-M$.
%The result follows from Propositions \ref{propn:main_technical} and \ref{prop:HFK} in the same way as \cite[Theorem 1.1]{JMZ} follows from \cite[Proposition 4.1]{JMZ}.
%The only difference is that we use the unorientable genus, defined as $\gamma(\Sigma) = -\chi(\Sigma) = b-m-M$.
%
%The $+1$ in Equation \eqref{eq:U_main} is a direct consequence of the extra $U$ factor on the right hand side of Equation \eqref{eq:HFK}.
\end{proof}
%
%Based on the parallel with Equation \eqref{eq:I_main} and \cite[Theorem 1.2]{JMZ}, and on upcoming work of Wong (see Remark \ref{rem:Wong}), we formulate the following conjecture.
%
%\begin{conjecture}
%\label{conj:+1}
%The $+1$ on the right hand side of Equation \eqref{eq:U_main} can be removed.
%\end{conjecture}

We now focus on the proofs of the Corollaries from the introduction.
Corollary \ref{cor:ribbon} follows immediately from Theorem \ref{thm:main} by setting $M=0$, so we move directly to the following corollary, about the refined unoriented cobordism distance.

Recall that for a cobordism $\Sigma$ in $I \times S^3$ from $K_1$ to $K_2$, we define
\[
|\Sigma| = \max\set{m,M} -\chi(\Sigma),
\]
and that the \emph{refined unorientable cobordism distance} between two knots $K_1$ and $K_2$ is given by 
\[
d_u^r(K_1, K_2) = \min\set{|\Sigma|},
\]
where $\Sigma$ varies over all connected unorientable cobordisms and oriented concordances from $K_1$ to $K_2$.

%\begin{definition}
%The \emph{standard all-surface cobordism distance} between two knots $K_1$ and $K_2$ is given by 
%\begin{align*}
%d_a(K_1, K_2) &= \min\set{b_1(\Sigma)-1\,\middle|\,\text{$\Sigma$ connected cobordism in $I \times S^3$ from $K_1$ to $K_2$}}.
%\end{align*}
%\end{definition}
%
%\begin{definition}
%The \emph{refined all-surface cobordism distance} between two knots $K_1$ and $K_2$ is given by 
%\[
%d_a^r(K_1, K_2) = \min\set{|\Sigma|\,\middle|\,\text{$\Sigma$ connected cobordism in $I \times S^3$ from $K_1$ to $K_2$}}.
%\]
%\end{definition}
%
%Like the previous cases, $d_a$ defines a distance on the knot concordance group, while $d_a^r$ defines a distance on the set of knots. Moreover, it is immediate to check that
%\begin{align*}
%d_a &= \min\set{d_o, d_u}, & d_a^r &= \min\set{d_o^r, d_u^r}.
%\end{align*}

{\renewcommand{\thetheorem}{\ref{cor:dur}}
\begin{corollary}
%\label{cor:dur}
If $K_1$ and $K_2$ are knots in $S^3$, then
\[
|\OrdI(K_1) - \OrdI(K_2)| \leq d_u^r(K_1,K_2)
\]
and
\[
|\OrdU(K_1) - \OrdU(K_2)| \leq d_u^r(K_1,K_2).
\]
\end{corollary}
\addtocounter{theorem}{-1}
}
\begin{proof}
The proof follows closely that of \cite[Corollary 1.5]{JMZ}.
Given a cobordism $\Sigma$ from $K_1$ to $K_2$ with $M$ maxima and $m$ minima of the kind considered in the definition of $d_u^r$, by Theorem \ref{thm:main} (if $\Sigma$ is unorientable) and Remark \ref{rem:all} (if $\Sigma$ is an orientable concordance) we have
\[
\Ord(K_1) \leq \max\{\Ord(K_2), M\} - \chi(\Sigma) \leq \Ord(K_2) + M -\chi(\Sigma),
\]
where $\Ord$ is either $\OrdI$ or $\OrdU$. From here we get
\[
\Ord(K_1) - \Ord(K_2) \leq M -\chi(\Sigma) \leq \max\{m,M\} - \chi(\Sigma),
\]
and we conclude by exchanging the roles of $K_1$ and $K_2$, and taking the minimum on the right hand side.
\end{proof}

%In Corollary \ref{cor:example} we will show that for all $\gamma, m>0$ there are knots $K_{\gamma, m}$ with
%\[
%d_u(K_{\gamma,m}, U_1) = \gamma, \qquad d_u^r(K_{\gamma,m}, U_1) = \gamma+m.
%\]
%
%\begin{remark}
%Of course Corollary \ref{cor:dur} works also with $d_a^r$ instead of $d_u^r$ (see Remark \ref{rem:comparison}).
%However, for the case of orientable surfaces, one can use the work of Juh\'asz--Miller--Zemke~\cite{JMZ}, which should give better bounds.
%\end{remark}
%
%\begin{remark}
%If Conjecture \ref{conj:+1} is true, then the $+1$ on the right hand side of the displayed equation of Corollary \ref{cor:dur} can be removed.
%\end{remark}

Recall that the \emph{unoriented band-unlinking number} $\ubuln(K)$ of a knot $K$ in $S^3$ is defined as the minimum number of (orientable or unorientable) band surgeries that turn $K$ into an unlink.

{\renewcommand{\thetheorem}{\ref{cor:ubuln}}
\begin{corollary}
For a knot $K$ in $S^3$,
\[
\OrdI(K) \leq \ubuln(K) \qquad \text{and} \qquad \OrdU(K) \leq \ubuln(K).
\]
\end{corollary}
\addtocounter{theorem}{-1}
}
\begin{proof}
The proof is similar to the one of \cite[Corollary 1.6]{JMZ}.
If $b = \ubuln(K)$, one can build a cobordism $\Sigma$ from $K$ to the unknot $U$ with $b$ saddles and $M$ local maxima, by attaching $b$ bands to $K$ to get an $(M+1)$-component unlink and then capping off $M$ components of the latter.
By applying Theorem \ref{thm:main}, and Remark \ref{rem:all} if necessary (i.e., if $\Sigma$ is orientable), we get (for $I^\sharp$ or $\HFKum$)
\[
\Ord(K) \leq \max\{\Ord(U), M\} - \chi(\Sigma) = M -\chi(\Sigma),
\]
since the unknot has vanishing torsion order in both $I^\sharp$ and $\HFKum$.
We conclude by noticing that $\chi(\Sigma) = M-b$.
\end{proof}

%\begin{remark}
%Using Remark \ref{rem:all}, in the orientable case one can also prove that $\OrdU(K) \leq \buln(K)$, without the $+1$ on the right hand side. However, since $\OrdU(K) \leq \Order_{u,v}^{\mathrm{Chain}}(K)$, this bound can never be better than the one in \cite[Proposition 7.4.(1)]{JMZ}.
%\end{remark}

%\begin{remark}
%\label{rem:Wong}
%In an upcoming paper~\cite{Wong}, Wong will show with different methods that if there is a cobordism $\Sigma \subset I \times S^3$ from $K_1$ and $K_2$ (no matter whether orientable or unorientable) with $m$ minima, $b$ saddles, and $M$ maxima, then
%\[
%|\OrdU(K_1) - \OrdU(K_2)| \leq m+b+M.
%\]
%Since the unlink has vanishing torsion order, this would strengthen Corollary \ref{cor:ubuln}, allowing to strike the $+1$ on the right hand side.
%\end{remark}
\section{Examples}
\label{sec:examples}

\begin{lemma}
\label{lem:TOtorusknots}
For the torus knot $T_{n,n+1}$, $\OrdU(T_{n,n+1}) = \floor{n/2}$.
%\MM{Sherry knows about $T_{n,n+2}$ and $T_{n, n+3}$ too.}
\end{lemma}

\begin{proof}
Any torus knot is an $L$-space knot, so its Alexander polynomial determines the full knot Floer complex $\CFK^\infty$ up to chain homotopy equivalence, and in turn the unoriented knot Floer homology. See \cite{OSz-Lens, Upsilon, JMZ}.

If $K$ is an $L$-space knot, its Alexander polynomial takes the form
\begin{equation}
\label{eq:Alex}
\Delta_K(t) = \sum_{k=0}^{2l} (-1)^k t^{\a_k}
\end{equation}
for a decreasing sequence of integers $\a_0, \ldots, \a_{2l}$. Let $d_1, \ldots, d_{2l}$ denote the gaps, i.e., $d_k = \a_k-\a_{k-1}$.
Then the full knot Floer complex is (up to chain homotopy equivalence) a staircase $\F[U, U^{-1}]$-module, generated by $\x_0, \ldots, \x_{2l}$, with
\[
\de \x_{2k} = 0, \qquad \de\x_{2k+1} = \x_{2k}+\x_{2k+2}.
\]
Moreover, the filtration over $\Z \oplus \Z$ is determined up to an overall shift by the following properties:
\begin{itemize}
\item the element $\x_{2k+1}$ has the same $j$-filtration as $\x_{2k}$, but the $i$-filtration differs by $d_{2k+1}$;
\item the element $\x_{2k+1}$ has the same $i$-filtration as $\x_{2k+2}$, but the $j$-filtration differs by $d_{2k+2}$.
\end{itemize}
Then, the unoriented knot Floer complex $\CFKum(K)$ (up to chain homotopy equivalence) is generated over $\F[U]$ by $\y_0, \ldots, \y_{2l}$, with differential
\[
\de \y_{2k} = 0, \qquad \de\y_{2k+1} = U^{d_{2k+1}}\cdot\y_{2k}+U^{d_{2k+2}}\cdot\y_{2k+2}.
\]
In the language of \cite{DM}, this is a \emph{standard complex associated to a graded root}. Graded roots were introduced by N\'emethi~\cite{Nemethi} to study $\HF^+$ of plumbed 3-manifolds. We instead consider the `upside-down' graded roots as in \cite{DM}, which are used to describe $\HF^-$. Note that the generators that we call $\y_0, \ldots, \y_{2l+1}$ were called $v_1, \a_1, v_2, \a_2, \ldots, \a_{n-1}, v_n$ in \cite{DM}. The numbers $d_i$ determine the graded root up to an overall shift: the (relative) grading is given by
\[
\chi(\y_{2k}) - \chi(\y_{2k+1}) = d_{2k+1}, \qquad
\chi(\y_{2k+2}) - \chi(\y_{2k+1}) = d_{2k+2}.
\]

We now determine the numbers $d_i$ in the case of the torus knot $T_{n,n+1}$.
Recall that the Alexander polynomial of $T_{p,q}$ is
\[
\Delta_{p,q}(t) = \frac{(t^{pq}-1)\cdot(t-1)}{(t^p-1)\cdot(t^q-1)}.
\]
The coefficients of $\Delta_{p,q}(t)$ have been computed in the general case (see for example \cite[(1.6) and (2.16)]{Cyclotomic1} or \cite{Cyclotomic2}).
In our case $p=n$ and $q=n+1$, and $\Delta_{p,q}$ is simple enough to be computed explicitly. After simplifying
\[
\Delta_{n, n+1} = \frac{(x^n)^n + (x^n)^{n-1} + \cdots x^n + 1}{x^n + x^{n-1} + \cdots + x + 1},
\]
one can carry out the long division explicitly and find that $\Delta_{n, n+1}$ is in the form of Equation \eqref{eq:Alex}, with
\[
(d_1, d_2, d_3, d_4, \ldots, d_{2l-1}, d_{2l}) = (1, n-1, 2, n-2, \ldots, n-1, 1).
\]

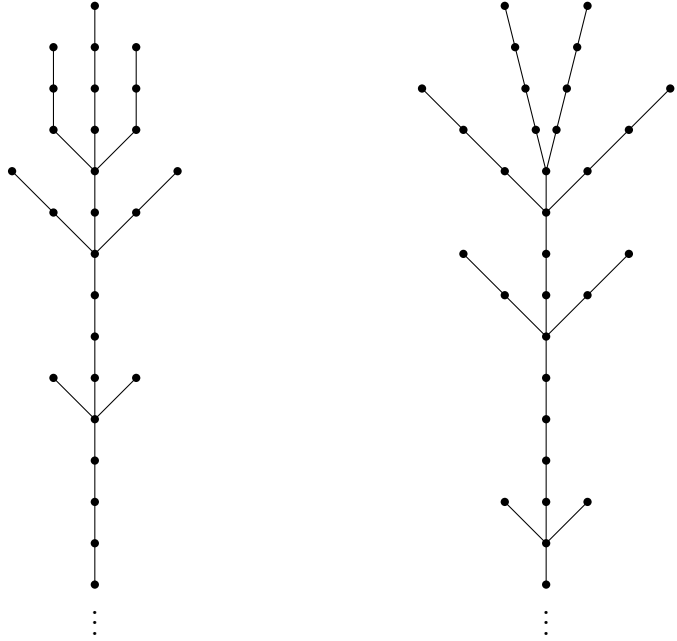
\begin{figure}
%\begin{center}
\begin{tikzpicture}

\def\u{1.1cm}
\def\circleradius{\u/20}
\def\gap{\u/2}

\newcommand{\nd}[1]{
\fill[black] #1 circle (\circleradius)
}

\begin{scope}[xshift=-3cm]

% main tower
\foreach \x in {0,-1,...,-13}
{
\nd{(0,\x*\gap)};
\draw (0,\x*\gap) -- (0,\x*\gap-\gap);
}
\nd{(0,-14*\gap)};
\draw (0,-14*\gap) node [anchor=north]  {$\vdots$};

\foreach \SIGN in {1,-1}
{

% first branches
\foreach \x in {-1,-2}
{
\nd{(\SIGN*\gap, \x*\gap)};
\draw (\SIGN*\gap, \x*\gap) -- (\SIGN*\gap, \x*\gap-\gap);
}
\nd{(\SIGN*\gap, -3*\gap)};
\draw (\SIGN*\gap, -3*\gap) -- (0, -4*\gap);

% second branches
\begin{scope}[yshift=-6*\gap]
\foreach \x in {1,2}
{
\nd{(\SIGN*\x*\gap,\x*\gap)};
\draw (\SIGN*\x*\gap-\SIGN*\gap,\x*\gap-\gap) -- (\SIGN*\x*\gap,\x*\gap);
}
\end{scope}

% third branches
\begin{scope}[yshift=-10*\gap]
\nd{(\SIGN*\gap,\gap)};
\draw (0,0) -- (\SIGN*\gap,\gap);
\end{scope}

}

\end{scope}

\begin{scope}[xshift=3cm]

% main tower
\foreach \x in {-4,-5,...,-13}
{
\nd{(0,\x*\gap)};
\draw (0,\x*\gap) -- (0,\x*\gap-\gap);
}
\nd{(0,-14*\gap)};
\draw (0,-14*\gap) node [anchor=north]  {$\vdots$};

\foreach \SIGN in {1,-1}
{

\begin{scope}[yshift=-4*\gap]
% top branches
\foreach \x in {1,2,...,4}
{
\nd{(\SIGN*\gap*\x/4, \x*\gap)};
\draw (\SIGN*\gap*\x/4, \x*\gap) -- (\SIGN*\gap*\x/4-\SIGN*\gap/4, \x*\gap-\gap);
}
\end{scope}

% second branches
\begin{scope}[yshift=-5*\gap]
\foreach \x in {1,2,3}
{
\nd{(\SIGN*\x*\gap,\x*\gap)};
\draw (\SIGN*\x*\gap-\SIGN*\gap,\x*\gap-\gap) -- (\SIGN*\x*\gap,\x*\gap);
}
\end{scope}

% third branches
\begin{scope}[yshift=-8*\gap]
\foreach \x in {1,2}
{
\nd{(\SIGN*\x*\gap,\x*\gap)};
\draw (\SIGN*\x*\gap-\SIGN*\gap,\x*\gap-\gap) -- (\SIGN*\x*\gap,\x*\gap);
}
\end{scope}

% last branches
\begin{scope}[yshift=-13*\gap]
\nd{(\SIGN*\gap,\gap)};
\draw (0,0) -- (\SIGN*\gap,\gap);
\end{scope}

}

\end{scope}

\end{tikzpicture}
%\resizebox{\textwidth}{!}{
%\includegraphics{gradedroots.jpg}
%}
\caption{The figure shows the graded roots homotopy equivalent to $\CFKum(T_{7,8})$ and $\CFKum(T_{8,9})$ respectively.
Each dot in the figure denotes a generator of the complex over $\mathbb{F}$, and the edges encode the $U$-action: for a dot $\x$, $U \cdot \x$ is the dot where you get by following the edge exiting from the bottom of the dot $\x$. The height of the dot denotes its (relative) Maslov grading, and the $U$ action decreases the Maslov grading by $2$.
Note that when $n$ is odd (e.g.~$T_{7,8}$) there is one branch of length $\ceil{n/2}$ and two branches of length $\floor{n/2}$, whereas when $n$ is even (e.g.~$T_{8,9}$) there are two branches of length $n/2$.}
\label{fig:gradedroots}
%\end{center}
\end{figure}
%\MM{Is Figure \ref{fig:gradedroots} any clearer now?}

From this one can check that the graded root has a picture with $n$ branches, of lengths $1$, $2$, \ldots, $2$, $1$. See Figure \ref{fig:gradedroots}. The longest branch is in the middle, of length $\ceil{n/2}$. This is also the top graded branch, so it generates the infinite tower. Thus, the order of $\HFKum$ is given by the next longest branch, which has length $\floor{n/2}$. Thus, $\OrdU(T_{n,n+1}) = \floor{n/2}$.
\end{proof}

We now restrict the attention to the torus knots of the form $T_{2r-1,2r}$.

\begin{remark}
\label{rem:torusknots}
Batson~\cite{Batson} first proved that $\gamma_4(T_{2r-1,2r}) = r-1$. This can be proved with any of the bounds from \cite{Batson, tHFK, GM} (for $T_{2r-1,2r}$ or $\m{T_{2r-1,2r}}$), which all give the same sharp obstruction.
We choose to use $\upsilon$ from \cite{tHFK} because it is an additive quantity, like the knot signature. In \cite[Theorem 1.2]{tHFK}, Ozsv\'ath--Stipsicz--Szab\'o proved that for a knot $K$ in $S^3$
\begin{equation}
\label{eq:OSSz}
\gamma_4(K) \geq \upsilon(K) - \frac{\sigma(K)}{2}.
\end{equation}
Batson~\cite{Batson} computed that $\sigma(T_{2r-1,2r}) = -2r^2+2$, and by \cite[Theorem 1.3]{tHFK} one can compute $\upsilon(T_{2r-1,2r}) = -r^2+r$. Thus,
\begin{equation}
\label{eq:sigmaupsilonT}
\upsilon(T_{2r-1,2r}) - \frac{\sigma(T_{2r-1,2r})}2 = r-1.
\end{equation}
%{
%\color{olive}
%\MM{Keep or strike this olive paragraph?}
%For completeness, we quickly outline the computation of $\upsilon(T_{2r-1,2r})$ here.
%From the computations in Lemma \ref{lem:TOtorusknots}, we compute the $\a_k$ in Equation \eqref{eq:Alex}:
%\[
%\a_{2k} = (2r-1)\cdot(r-k-1), \qquad \a_{2k+1} = \a_{2k} - (k+1).
%\]
%Ozsv\'ath--Stipsicz--Szab\'o~\cite{tHFK} define the quantity $m_{2k}$ recursively, by setting
%\[
%m_0 = 0, \qquad m_{2k+2} = m_{2k} - 2(\a_{2k}-\a_{2k+1}),
%\]
%from which we compute that $m_{2k} = -2\frac{k(k+1)}2$ in the case of $T_{2r-1,2r}$. The maximum of
%\[
%m_{2k}-\a_{2k} = -k^2-k-(2r-1)\cdot(r-1-k)
%\]
%is achieved for $k=r-1$. By \cite[Theorem 1.3]{tHFK} we have
%\[
%\upsilon(T_{2r-1,2r}) = \max\set{m_{2k}-\a_{2k}} = -r^2+r.
%\]
%}
\end{remark}

%\begin{corollary}
%\label{cor:example}
%For all $\gamma\geq1$ and $m\geq1$, there exists a knot $K_{\gamma, m}$ with $\gamma_4(K_{\gamma, m}) = \gamma$, and such that each (possibly non-orientable) surface $\Sigma \subset B^4$ with $\de\Sigma = K_{\gamma, m}$ and $b_1(\Sigma) = \gamma$ has at least $m$ local minima (with respect to the radial function).
%
%In fact, $d_u(K_{\gamma, m}, U_1) = \gamma$ and $d_u^r(K_{\gamma, m}, U_1) \geq \gamma+m$.
%\end{corollary}
%
%{\color{olive}

We now restate and prove Corollary \ref{cor:example} from the introduction.

{\renewcommand{\thetheorem}{\ref{cor:example}}
\begin{corollary}
%\label{cor:example}
For all $\gamma\geq1$ and $m\geq1$, there exists a knot $K_{\gamma, m}$ with $d_u(K_{\gamma, m}, U_1) = \gamma_4(K_{\gamma, m}) = \gamma$, and such that $d_u^r(K_{\gamma, m}, U_1) \geq \gamma+m$.

Thus, each unorientable surface $\Sigma \subset B^4$ with $\de\Sigma = K_{\gamma, m}$ and $\gamma(\Sigma) = \gamma$ has at least $m$ local minima (with respect to the radial function).
\end{corollary}
\addtocounter{theorem}{-1}
}
%}
\begin{proof}
Let $K_{\gamma, m} = T_{2r-1,2r} \# \m{T_{2s-1, 2s}}$, for $r = \gamma+m$ and $s = m$.
By Equation \eqref{eq:OSSz}, the additivity of $\sigma$ and $\upsilon$, and Equation \eqref{eq:sigmaupsilonT}, we have that
\begin{align*}
\gamma_4(K_{\gamma, m}) 
&\geq \left(\upsilon(T_{2r-1,2r}) - \frac{\sigma(T_{2r-1,2r})}{2}\right) - \left(\upsilon(T_{2s-1,2s}) - \frac{\sigma(T_{2s-1,2s})}{2}\right) \\
&= (r-1) - (s-1) = r-s = \gamma.
\end{align*}
On the other hand, Batson showed that there is a sequence $r-s$ unoriented band surgeries from $T_{2r-1,2r}$ to $T_{2s-1, 2s}$~\cite[Figure 7]{Batson}. Thus, we get a sequence of $r-s$ unoriented band surgeries from $K_{\gamma, m}$ to $T_{2s-1, 2s} \# \m{T_{2s-1, 2s}}$, which is slice, so $\gamma_4(K_{\gamma, m}) = r-s$.

Now let $\Sigma$ be a (possibly non-orientable) surface $\Sigma \subset B^4$ with $\de\Sigma = K_{\gamma, m}$ and $b_1(\Sigma) = \gamma$. Theorem \ref{thm:main} gives a lower bound on the number of local minima. More precisely, if $\Sigma$ has $n$ local minima, by removing a small ball from $B^4$ we get a cobordism from $K_{\gamma, m}$ to the unknot $U_1$ with $M = n-1$ maxima (note that the cobordism is upside down, so the minima are turned into maxima, and one of them disappears when we remove the ball). Thus, Theorem \ref{thm:main} implies that
\[
\OrdU(K_{\gamma, m}) \leq (n-1) + (r - s) = n-s +r - 1.
\]
We also know that
\[
\OrdU(K_{\gamma, m}) = \max\set{\OrdU(T_{2r-1, 2r}), \OrdU(T_{2s-1,2s})} = r-1
\]
by Remark \ref{rem:propertiesHFLum} (points \eqref{it:Kunneth} and \eqref{it:mirror}) and Lemma \ref{lem:TOtorusknots}, so we get
\[
n \geq s = m.
\]
The statement about $d_u$ and $d_u^r$ follows from the computation of $\gamma_4(K_{\gamma, m})$ above and from Corollary \ref{cor:dur}.
\end{proof}

%{\color{olive}
%\MM{If $m=1$ it seems that the bound is sharp, which is surprising and confusing?}
\begin{remark}
We do not know if the bound on $d_u^r(K_{\gamma, m}, U_1)$ and on the number of minima of $\Sigma$ in Corollary \ref{cor:example} is sharp on the knots used in the proof of the corollary. Recall that we set
\[
K_{\gamma, m} := T_{2r-1,2r} \# \m{T_{2s-1, 2s}}
\]
for $r = \gamma+m$ and $s = m$. Batson showed that with $\gamma$ bands we can get to $K_{0, m} = T_{2s-1,2s} \# \m{T_{2s-1, 2s}}$, and Juh\'asz--Miller--Zemke showed that $K_{0,m}$ bounds a ribbon disc with $2m-1$ local minima. Thus,
\[
d_u^r(K_{\gamma, m}, U_1) \leq \gamma + 2m - 2.
\]
We conjecture that this inequality is actually an equality.
%We believe that the reason why we do not get a sharp bound is twofold:
%\begin{itemize}
%\item the $+1$ on the right hand of Equation \eqref{eq:U_main};
%\item for slice discs, which are orientable, the bound from \cite{JMZ} is better.
%\end{itemize}
\end{remark}

\bibliography{GM}
\bibliographystyle{alpha}

%\newpage
\end{document}